\documentclass[12pt]{amsart}
\usepackage[utf8]{inputenc}
\usepackage[T1]{fontenc}
\usepackage{amsfonts}
\usepackage[leqno]{amsmath}
\usepackage{accents}
\usepackage{amssymb, comment, float}
\usepackage[dvipsnames]{xcolor}
\usepackage[foot]{amsaddr}
\usepackage[shortlabels]{enumitem}
\setlist[itemize]{leftmargin=11mm}
\setlist[enumerate]{leftmargin=11mm}
\usepackage{mathdots}
\usepackage[a4paper, footskip=0.5in,  headheight = 0.5in, top=1.25in, bottom=1.25in,  right=1in,  left=1in]{geometry}
\usepackage{hyperref}
\hypersetup{
colorlinks,
linkcolor=black,
urlcolor=Blue,
citecolor=black,}
\usepackage{cleveref}
\usepackage[center]{caption}
\usepackage{eufrak}
\usepackage{marvosym}     
\usepackage{latexsym}
\usepackage[nodayofweek]{datetime}
\usepackage{tikz}
\usepackage{subfig}
\usepackage{thm-restate}
\usepackage{verbatim}
\usepackage[none]{hyphenat} 

\newtheorem{statement}{statement}[section]
\newtheorem{theorem}[statement]{Theorem}
\newtheorem{lemma}[statement]{Lemma}
\newtheorem{conjecture}[statement]{Conjecture}
\newtheorem{corollary}[statement]{Corollary}

\newtheorem{observation}[statement]{Observation}

\DeclareMathOperator{\tw}{{\text{\sf{tw}}}}
\DeclareMathOperator{\pw}{{\text{\sf{pw}}}}
\DeclareMathOperator{\pa}{{\text{\sf{path}}}-\alpha}
\DeclareMathOperator{\dist}{{\text{\sf{dist}}}}
\DeclareMathOperator{\diam}{{\text{\sf{diam}}}}
\DeclareMathOperator{\ex}{{\text{\sf{ex}}}}
\DeclareMathOperator{\bin}{{\text{\sf{B}}}}

\def\dd{\hbox{-}}   

\newcommand{\mf}{\mathfrak}
\newcommand{\mca}{\mathcal}
\newcommand{\poi}{\mathbb{N}} 

\newcounter{tbox}
\newcommand{\sta}[1]{%
  \par\addvspace{0.5\baselineskip}
  \refstepcounter{tbox}%
  \noindent(\thetbox)\,{\em #1}\par%
  \addvspace{0.3\baselineskip}%
}

\title{Polynomial bounds for pathwidth}

\author{Sepehr Hajebi$^{\dagger}$}
\email{sepehr.hajebi@utoronto.ca}

\thanks{$^{\dagger}$ Department of Mathematics, University of Toronto, Toronto, Ontario, Canada}
\date {\today}

\begin{document}

\maketitle

\begin{abstract}
Dallard, Milani\v{c}, and \v{S}torgel conjectured that for a hereditary graph class $\mca{G}$, if there is some function $f:\poi\to\poi$ such that every graph $G\in \mca{G}$ with clique number $\omega(G)$ has treewidth at most $f(\omega(G))$, then there is a polynomial function $f$ with the same property. Chudnovsky and Trotignon refuted this conjecture in a strong sense, showing that neither polynomial nor any prescribed growth can be guaranteed in general. Here we prove that, in stark contrast, the analog of the Dallard-Milani\v{c}-\v{S}torgel conjecture for \textit{pathwidth} is true: For every hereditary graph class $\mca{G}$, if the pathwidth of every graph in $\mca{G}$ is bounded by some function of its clique number, then the pathwidth of every graph in $\mca{G}$ is bounded by a polynomial function of its clique number.
\end{abstract}
 
\section{Introduction}
\subsection{Background and the (weaker) main result.}\label{sec:background} We write $\poi$ for the set of all positive integers, and $2^{X}$ for the power set of a set $X$. In this paper, a graph $G=(V(G),E(G))$ consists of a finite vertex set $V(G)$ and an edge set $E(G)$ whose elements are $2$-subsets of $V(G)$; thus, ``loops'' and ``parallel edges'' are not allowed. A \textit{class} is a set of graphs taken up to isomorphism. See \cite{diestel} for standard graph-theoretic terminology.

Let $G$ be a graph. A \textit{tree decomposition} of $G$ is a pair $\mf{t}=(T,\beta)$ where $T$ is a tree and $\beta:V(T)\rightarrow 2^{V(G)}$ is a function such that for every $v\in V(G)$, the subgraph of $T$ induced by $\{x\in V(T):v\in \beta(x)\}$ is non-null and connected, and for every $uv\in E(G)$, there exists $x\in V(T)$ such that $u,v\in \beta(x)$. We say that $\mf{t}$ is a \textit{path decomposition} of $G$ if $T$ is a path. The \textit{treewidth} of $G$, denoted $\tw(G)$, is the minimum $w\in \poi\cup \{0\}$ for which $G$ has a tree decomposition $(T,\beta)$ such that $|\beta(x)|\leq w+1$ for all $x\in V(T)$. The \textit{pathwidth} of $G$, denoted $\pw(G)$, is defined analogously using path decompositions. 
\medskip

In two seminal results \cite{GMI,GMV} from their Graph Minors project in the 1980s, Robertson and Seymour proved that, for minor-closed classes, bounded pathwidth is equivalent to excluding a forest, and bounded treewidth is characterized by excluding a planar graph.

\begin{theorem}[Robertson and Seymour \cite{GMI}]\label{thm:RSGMI}
Let $\mca{G}$ be a minor-closed class. Then every graph in $\mca{G}$ has bounded pathwidth if and only if $\mca{G}$ excludes a forest.
\end{theorem}
\begin{theorem}[Robertson and Seymour \cite{GMV}]\label{thm:RSGMV}
Let $\mca{G}$ be a minor-closed class. Then every graph in $\mca{G}$ has bounded treewidth if and only if $\mca{G}$ excludes a planar graph.
\end{theorem}

In recent years, there has been tremendous interest in characterizing bounded \mbox{pathwidth} and treewidth in \textit{hereditary} classes -- those closed under taking induced \mbox{subgraphs}. In joint work with Chudnovsky and Spirkl \cite{tw18}, we resolved this inquiry for pathwidth. Despite substantial effort, the analogous result for treewidth remains out of reach.

On the other hand, there are hereditary classes of unbounded treewidth in which all it takes to bound the treewidth is to satisfy an obvious necessary condition: excluding a complete graph (recall \cite{diestel} that $\tw(K_{t})=\pw(K_{t})=t-1$ for every $t\in \poi$). Dallard, Milani\v{c}, and \v{S}torgel \cite{DMS} made a bold conjecture that the treewidth bound in this scenario {\sl never} needs to grow much faster than the size of the excluded complete graph. To state their conjecture, we first need to give a few definitions.

A \textit{clique} in a graph $G$ is a set of pairwise adjacent vertices, and a \textit{stable set} in $G$ is a set of pairwise nonadjacent vertices. The \textit{clique number} $\omega(G)$ of $G$ is the largest cardinality of a clique in $G$. A class $\mca{G}$ is (\textit{polynomially}) \textit{$(\tw,\omega)$-bounded} if there is a (polynomial) function $f:\poi\rightarrow \poi$ -- called a \textit{$(\tw,\omega)$-bounding function for $\mca{G}$} -- such that $\tw(G)\leq f(\omega(G))$ for every $G\in \mca{G}$. For instance, the class of chordal graphs is $(\tw,\omega)$-bounded, whereas the class of complete bipartite graphs is not (recall \cite{diestel} that $\tw(G)=\omega(G)-1$ for every chordal graph $G$, and $\tw(K_{t,t})=\pw(K_{t,t})=t$ for all $t\in \poi$).
\medskip

The following was conjectured\footnote{In \cite{DMS}, \Cref{conj:DMS} appears as Question~8.4, but it is immediately followed by their Conjecture~8.5 suggesting that an even stronger statement is true. We therefore believe it is safe to call it a ``conjecture.''} in \cite{DMS}:

\begin{conjecture}[Dallard, Milani\v{c}, \v{S}torgel \cite{DMS}]\label{conj:DMS}
     Every hereditary $(\tw,\omega)$-bounded class is polynomially $(\tw,\omega)$-bounded.
\end{conjecture}

Chudnovsky and Trotignon \cite{CT} showed that \Cref{conj:DMS} fails spectacularly:

\begin{theorem}[Chudnovsky and Trotignon {\cite[Theorem~5.1]{CT}}]\label{thm:CT}
    For every function $g:\poi\rightarrow \poi$, there is a hereditary $(\tw,\omega)$-bounded class $\mca{G}$ such that every $(\tw,\omega)$-bounding function $f$ for $\mca{G}$ satisfies $f(t)>g(t)$ for all $t\in \poi\setminus \{1\}$.
\end{theorem}

The analog of \Cref{conj:DMS} for pathwidth, however, turns out to be true, and that is a result we prove in this paper. A class $\mca{G}$ is (\textit{polynomially}) \textit{$(\pw,\omega)$-bounded} if there is a (polynomial) function $f:\poi\rightarrow \poi$ -- called a \textit{$(\pw,\omega)$-bounding function for $\mca{G}$} -- such that $\pw(G)\leq f(\omega(G))$ for every $G\in \mca{G}$.

\begin{theorem}\label{thm:main}
   Every hereditary $(\pw,\omega)$-bounded class is polynomially $(\pw,\omega)$-bounded.
\end{theorem}

\Cref{thm:main} has sweeping algorithmic consequences: Combined with \cite[Theorem~1.1]{logDMS} (see also the two paragraphs preceding that result), it implies that the \textsc{Maximum Weight Independent Set} problem, along with numerous other problems that are {\sf NP}-hard in general, admits quasi-polynomial-time algorithms in every hereditary $(\pw,\omega)$-bounded class. In fact, one may deduce the stronger result that graphs from any hereditary $(\pw,\omega)$-bounded class have ``polylogarithmic path independence number.'' Given a graph $G$, the \textit{path independence number} of $G$, denoted $\pa(G)$, is the smallest integer $s\geq 0$ for which $G$ admits a path decomposition $(T,\beta)$ where $|\beta(x)\cap S|\leq s$ for every $x\in V(T)$ and every stable set $S$ in $G$. It is straightforward to deduce the following from \Cref{thm:main} and the aforementioned result of \cite{logDMS} (we omit the details). All logarithms in this paper are base 2.

\begin{corollary}\label{cor:pathalpha}
    For every hereditary $(\pw,\omega)$-bounded class $\mca{G}$, there exist $c,d\in \mathbb N$ such that every graph $G\in \mca{G}$ on $n>1$ vertices satisfies $\pa(G)\leq c \log^d n$.
\end{corollary}

\noindent Notably, a recent construction from \cite{awesome} shows that the polylogarithmic bound is best possible up to optimizing $c,d\geq 1$.
\medskip

Viewed differently, it is perhaps more surprising in itself that pathwidth {\sl does} come with a ``free upgrade'' to polynomial bounds in the clique number. Statements such as \Cref{conj:DMS} tend to be far from true, as \Cref{conj:DMS} itself is no exception. Another example was a conjecture of Esperet \cite{esperet} that in a hereditary class, if the chromatic number is bounded by some function of the clique number, then it is bounded by a polynomial function; this too was later disproved \cite{espertdis2}. At least at first glance, pathwidth hardly seems ``friendly'' enough to defy this pattern. The only other graph parameter known to us to have this property in full generality and in a nontrivial way is the minimum degree \cite{GHdeg} (we say ``nontrivial'' because, after all, Ramsey's theorem \cite{multiramsey} already implies that in a hereditary class, if the number of vertices in every graph is bounded by some function of its clique number, then it is bounded by a polynomial function. There may be other results that could be interpreted in this vein).
\medskip

It is also worth noting that the degree of the $(\pw,\omega)$-bounding polynomial obtained in our proof of \Cref{thm:main} depends on the class. This is necessary:
\begin{observation}\label{obs:lowerbound}
    For every $d\in \mathbb N$, there exists a $(\pw,\omega)$-bounded class $\mca{G}_d$ such that every polynomial $(\pw,\omega)$-bounding function for $\mca{G}_d$ has degree at least $d$.
\end{observation}
\begin{proof}
   Let $\mca{G}_d$ be the class of all graphs with no stable set of cardinality $d+3$. Then $\mca{G}_d$ is hereditary $(\pw,\omega)$-bounded; in fact, by Ramsey's theorem \cite{multiramsey} (see \Cref{thm:ramsey}), we have $|V(G)|<(\omega(G)+1)^{d+2}$ for every $G\in \mca{G}_d$. Also, the lower bound for the off-diagonal Ramsey numbers from \cite[Theorem~1.1]{lowerramsey} implies that there exist $c_1,c_2>0$ depending on $d$ such that for every integer $t\geq c_2$, there is a $K_t$-free graph $G_t\in \mca{G}_d$ with $|V(G_t)|\geq c_1t^{d+1}$. In particular, for every integer $t\geq \max\{c_2,d+3\}$, we have $\chi(G_t)\geq |V(G_t)|/(d+2)> c_1t^{d}$, and so $\pw(G_t)\geq c_1t^{d}$ (recall that $\pw(G)\geq \tw(G)\geq \chi(G)-1$ for every graph $G$).
\end{proof}

\subsection{The stronger result and its consequences.} Our main result, \Cref{thm:mainjagged}, is a strengthening of \Cref{thm:main}. For $\zeta\in \poi$, we say that a graph $J$ is \textit{$\zeta$-jagged} if every induced subgraph $H$ of $J$ with $\pw(H)\geq 3$ has at least $\zeta$ vertices of degree two.

\begin{restatable}{theorem}{mainjagged}\label{thm:mainjagged}
   Let $\mca{G}$ be a hereditary class. Assume that $\mca{G}$ excludes a complete bipartite graph, and for some $\zeta\in \poi$, every $\zeta$-jagged graph in $\mca{G}$ has bounded pathwidth. Then $\mca{G}$ is polynomially $(\pw,\omega)$-bounded. 
\end{restatable}

In addition, \Cref{thm:mainjagged} also implies the pathwidth analogs of two other conjectures on $(\tw,\omega)$-boundedness. The first, due to the author \cite{chordalehf}, was disproved by Chudnovsky and Trotignon \cite{CT}. The second, due to Cocks \cite{cocks}, is (curiously) still open. A graph $G$ is \textit{$d$-degenerate}, for $d\in \poi$, if every induced subgraph of $G$ has a vertex of degree at most $d$. For instance, graphs of treewidth (or pathwidth) at most $d$ are $d$-degenerate \cite{diestel}. We say that $G$ is \textit{$H$-free}, for another graph $H$, if $G$ has no {\sl induced} subgraph isomorphic to $H$.

\begin{conjecture}[Hajebi \cite{chordalehf}; Disproved by Chudnovsky and Trotignon \cite{CT}]\label{conj:2-deg}
     A \mbox{hereditary} class $\mca{G}$ is $(\tw,\omega)$-bounded if (and only if) $\mca{G}$ excludes a complete bipartite graph and every $2$-degenerate graph in $\mca{G}$ has bounded treewidth.
\end{conjecture}

\begin{conjecture}[Cocks; Equivalent to {\cite[Conjecture~1.5]{cocks}}]\label{conj:cocks}
     A hereditary class $\mca{G}$ is $(\tw,\omega)$-bounded if (and only if) $\mca{G}$ excludes a complete bipartite graph, and for every non-complete graph $H\in \mca{G}$, the class of all $H$-free graphs in $\mca{G}$ is $(\tw,\omega)$-bounded.
\end{conjecture}

From \Cref{thm:mainjagged}, we deduce the following. Note that \Cref{thm:main} states that ``\ref{cor:main_b} implies \ref{cor:main_a}'':

\begin{corollary}\label{cor:main}
  The following are equivalent for every hereditary class $\mca{G}$.
  \begin{enumerate}[{\rm(a)}]
  \item\label{cor:main_z} $\mca{G}$ excludes a complete bipartite graph, and for some $\zeta\in \poi$, every $\zeta$-jagged graph in $\mca{G}$ has bounded pathwidth.
  \item\label{cor:main_a} $\mca{G}$ is polynomially $(\pw,\omega)$-bounded.
    \item\label{cor:main_b} $\mca{G}$ is $(\pw,\omega)$-bounded.
    \item\label{cor:main_c} Every $K_4$-free graph in $\mca{G}$ has bounded pathwidth\footnote{Line graphs of subdivided binary trees show that ``$K_4$-free'' here cannot be improved to $K_3$-free.}. 
    \item\label{cor:main_d} $\mca{G}$ excludes a complete bipartite graph, and every $2$-degenerate graph in $\mca{G}$ has bounded pathwidth.
    \item\label{cor:main_e} $\mca{G}$ excludes a complete bipartite graph, and for every non-complete graph $H\in \mca{G}$, the class of all $H$-free graphs in $\mca{G}$ is $(\pw,\omega)$-bounded.
  \end{enumerate}
\end{corollary}
\begin{proof}
By \Cref{thm:mainjagged}, \ref{cor:main_z} implies \ref{cor:main_a}. Clearly, \ref{cor:main_a} implies \ref{cor:main_b}, and \ref{cor:main_b} implies \ref{cor:main_c}. Moreover, \ref{cor:main_c} implies \ref{cor:main_d} because $2$-degenerate graphs are $K_4$-free, and \ref{cor:main_d} implies \ref{cor:main_z} because $1$-jagged graphs are $2$-degenerate. Note also that \ref{cor:main_b} implies \ref{cor:main_e}. We now show that \ref{cor:main_e} implies \ref{cor:main_z}. Suppose that $\mca{G}$ excludes a complete bipartite graph, and for every non-complete graph $H\in \mca{G}$, the class of all $H$-free graphs in $\mca{G}$ is $(\pw,\omega)$-bounded, but for all $\zeta,\xi\in \poi$, there is a $\zeta$-jagged graph $J_{\zeta,\xi}\in \mca{G}$ with $\pw(J_{\zeta,\xi})\ge \xi$. In particular, $J_{1,3}$ is not a complete graph. Let $\zeta^+=|V(J_{1,3})|+1$. Then $J_{\zeta^+,\xi}$ is $K_4$-free for every $\xi\in \poi$. Assume that $J_{\zeta^+,\xi}$ is not $J_{1,3}$-free for some $\xi\in \poi$. Let $H$ be an induced subgraph of $J_{\zeta^+,\xi}$ isomorphic to $J_{1,3}$. Then $\pw(H)=\pw(J_{1,3})\ge 3$. Since $J_{\zeta^+,\xi}$ is $\zeta^+$-jagged, $H$ has at least $\zeta^+$ vertices of degree two, and so $|V(J_{1,3})|=|V(H)|\geq \zeta^+= |V(J_{1,3})|+1$, a contradiction. We deduce that $J_{\zeta^+,\xi}$ is $J_{1,3}$-free for every $\xi\in \poi$. But now $J_{1,3}\in \mca{G}$ is a non-complete graph and the class of all $J_{1,3}$-free graphs in $\mca{G}$ is not $(\pw,\omega)$-bounded, because for all $\xi\in \poi$, $J_{\zeta^+,\xi}$ is $J_{1,3}$-free and $K_4$-free with $\pw(J_{\zeta^+,\xi})\geq \xi$, a contradiction.
\end{proof}

This paper is organized as follows. \Cref{sec:notation} introduces some notation and terminology. In \Cref{sec:gos}, we break down the proof of \Cref{thm:mainjagged} into three main parts, Theorems~\ref{thm:pwtotw}, \ref{thm:mainpolyblock}, and \ref{thm:mainpolysep}, and derive \Cref{thm:mainjagged} from these results (and others from the literature). We then prove Theorems~\ref{thm:pwtotw}, \ref{thm:mainpolyblock}, and \ref{thm:mainpolysep} in Sections~\ref{sec:pwtotw}, \ref{sec:mainpolyblock}, and \ref{sec:mainpolysep}, respectively.

\section{Preliminaries}\label{sec:notation}
For integers $k,k'$, we denote by $\{k,\ldots,k'\}$ the set of all integers that are at least $k$ and at most $k'$. For $k\in \poi\cup \{0\}$, we denote by $\binom{X}{k}$ the set of all $k$-subsets of a set $X$. 

Let $k\in \poi$ and let $P$ be a $k$-vertex path. We write $P = x_1 \dd \cdots \dd x_k$ to indicate that $V(P) = \{x_1, \dots, x_k\}$ and $E(P)=\{x_ix_{i+1}:i\in \{1,\ldots, k-1\}\}$. We call $x_1$ and $x_k$ the \textit{ends of $P$}, and say that $P$ is \textit{from $x_1$ to $x_k$}. The \textit{interior of $P$}, denoted $P^*$, is the (possibly null) path obtained from $P$ by removing $x_1,x_k$. For $x,x'\in V(P)$, the subpath of $P$ from $x$ to $x'$ is denoted by $x\dd P\dd x'$. The \textit{length} of $P$ is $|E(P)|=k-1$.

Let $G$ be a graph. A {\em path in $G$} is an induced subgraph of $G$ that is a path. For a set $\mca{P}$ of paths in $G$, we define $V(\mca{P})=\bigcup_{P\in \mca{P}}V(P)$. For $X\subseteq V(G)$, we denote by $G[X]$ and $G\setminus X$, respectively, the induced subgraphs of $G$ with vertex sets $X$ and $V(G)\setminus X$. We denote by $N_G(X)$ the set of all vertices in $V(G)\setminus X$ with at least one neighbor in $X$. When $X=\{x\}$, we write $G\setminus x$ and $N_G(x)$ instead of $G\setminus \{x\}$ and $N_G(\{x\})$. For $X,Y\subseteq V(G)$, we say that \textit{$X$ and $Y$ are anticomplete in $G$} if $X\cap Y=\varnothing$ and there is no edge in $G$ with an end in $X$ and an end in $Y$; in this case, we also say that \textit{$G[X]$ and $G[Y]$ are anticomplete in $G$}. If $X=\{x\}$, then we say that \textit{$x$ is anticomplete to $Y$ in $G$} to mean that $\{x\}$ and $Y$ are anticomplete in $G$.

Let $G$ be a graph. For vertices $u,v\in V(G)$, we denote by $\dist_G(u,v)$ the \textit{distance between $u$ and $v$ in $G$}; that is, the length of the shortest path in $G$ from $u$ to $v$. For $u\in V(G)$ and $\rho\in \poi\cup \{0\}$, we write $N^{\rho}_G(u)=\{v\in V(G):\dist_{G}(u,v)=\rho\}$. It follows that $N^{0}_G(u)=\{u\}$ and $N^{1}_G(u)=N_G(u)$. Also, note that:
\begin{itemize}
\item $N^{\rho}_G(u)=\varnothing$ for all sufficiently large $\rho$;
    \item if $G$ is connected, then $(N^{\rho}_G(u):\rho\in \poi\cup\{0\})$ is a partition of $V(G)$;
    \item for all $\rho\in \poi$, every vertex in $N^{\rho}_G(u)$ has at least one neighbor in $N^{\rho-1}_G(u)$.
    \item $N^{\rho}_G(u)$ and $N^{\rho'}_G(u)$ are anticomplete in $G$ for all $\rho,\rho'\in \poi\cup \{0\}$ with $|\rho-\rho'|>1$.
\end{itemize}

Recall that a \textit{minor} of a graph $G$ is a graph obtained from $G$ by removing vertices, removing edges, and contracting edges, and an \textit{induced minor} of $G$ is a graph obtained from $G$ by only removing vertices and contracting edges (in both cases, loops and parallel edges arising from contracting edges are automatically removed). Minors and induced minors may also be defined using ``models.'' For graphs $G$ and $H$, an \textit{$H$-model in $G$} is a tuple $\mf{m}=(X_v:v\in V(H))$ of pairwise disjoint connected (and non-null) induced subgraphs of $G$ such that for every $uv\in E(H)$, $X_u$ and $X_v$ are not anticomplete in $G$. We say that $\mf{m}$ is \textit{induced} if, in addition, for all distinct vertices $u,v\in V(H)$ with $uv\notin E(H)$, $X_u$ and $X_v$ are anticomplete in $G$. Note that $G$ has an (induced) minor isomorphic to $H$ if and only if there is an (induced) $H$-model in $G$. When $H=K_{p,q}$ for $p,q\in \poi$, we denote a $K_{p,q}$-model $\mf{c}$ in $G$ by $(A_1,\ldots, A_p;B_1,\ldots, B_q)$ to mean that:
\begin{itemize}
    \item $A_1,\ldots, A_p, B_1,\ldots, B_q$ are pairwise disjoint connected induced subgraphs of $G$;
    \item for all $i\in \{1,\ldots, p\}$ and $j\in \{1,\ldots, q\}$, $A_i$ and $B_j$ are not anticomplete in $G$; and
    \item if $\mf{c}$ is an induced $K_{p,q}$-model in $G$, then $A_1,\ldots, A_p$ are pairwise anticomplete in $G$, and $B_1,\ldots, B_q$ are pairwise anticomplete in $G$.
\end{itemize}
As a special case, we say that $\mf{c}$ is a \textit{$(p,q)$-constellation in $G$} if $\mf{c}$ is induced, we have $|V(A_1)|=\cdots=|V(A_p)|=1$, and each of $B_1,\ldots, B_q$ is a path in $G$.

A \textit{subdivision} of a graph $H$ is a graph $H'$ obtained from $H$ by replacing the edges of $H$ with
pairwise internally disjoint paths of nonzero length between the corresponding ends. For $d \in \poi\cup \{0\}$, we say that $H'$ is a \textit{$(\geq d)$-subdivision} of $H$ if each edge of $H$ is replaced by a path of length at least $d+1$, and that $H'$ is a \textit{$d$-subdivision} of $H$ if each edge of $H$ is replaced by a path of length exactly $d+1$. A \textit{proper subdivision} of $H$ is a $(\geq 1)$-subdivision of $H$. Note that if a graph $G$ has an (induced) subgraph isomorphic to a subdivision of $H$, then $G$ has an (induced) minor isomorphic to $H$. Moreover, if $|V(H)|=h\in \poi$ and a graph $G$ has an induced subgraph isomorphic to a proper subdivision of $K_h$, then $G$ has an induced subgraph isomorphic to a proper subdivision of $H$, and so $G$ has an induced minor isomorphic to $H$.

The \textit{line graph} of $H$ is the graph $L$ with vertex set $E(H)$ such that $e,f\in E(H)$ are adjacent in $L$ if and only if $e$ and $f$ share an end in $H$.

For $r\in \poi\cup \{0\}$, we denote by $\bin_{r}$ the (full) binary tree of radius $r$. It is well known \cite{GMI} that all subdivisions of $\bin_{r}$ and their line graphs have pathwidth at least $\lfloor\frac{r+1}{2}\rfloor$. Also, for $r\in \poi$, by the \textit{$r$-wall} we mean the $r$-by-$r$ hexagonal grid, and by the \textit{$r$-grid} we mean the $r$-by-$r$ square grid (see Figure~\ref{fig:binarywall}). It is straightforward to check that:

\begin{observation}\label{obs:wall}
    For all $r\in \poi$, the $(2r-1)$-wall has an induced minor isomorphic to the $r$-grid.
\end{observation}
\begin{figure}[t!]
    \centering
    \includegraphics[scale=0.6]{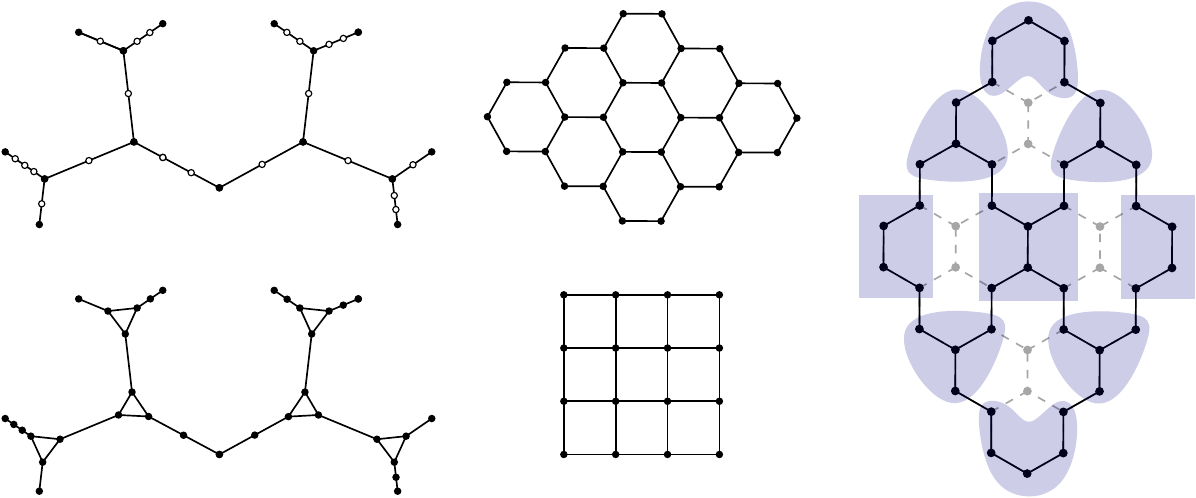}
    \caption{\raggedright Left: A subdivision of $\bin_3$ (top) and its line graph (bottom). Middle: The $4$-wall (top) and the $4$-grid (bottom). Right: Note that the $5$-wall contains an induced model of the $3$-grid (as in \Cref{obs:wall}).}
    \label{fig:binarywall}
\end{figure}

\section{Getting the good old stuff out of the way}\label{sec:gos}

There are three steps in the proof of \Cref{thm:mainjagged} that comprise the main contribution of this paper. The first is to prove that excluding a (binary) tree as an induced minor bounds the pathwidth by a polynomial function of the treewidth:

\begin{restatable}{theorem}{pwtotw}\label{thm:pwtotw}
    For all $\rho,\tau\in \poi$, every graph $G$ with $\tw(G)<\tau$ and no induced minor isomorphic to $\bin_{\rho}$ satisfies $\pw(G)< \tau^{6\rho+2}$. 
\end{restatable}

The second and third steps are, respectively, to prove Theorems~\ref{thm:mainpolyblock} and \ref{thm:mainpolysep} below; these results deal with ``separability'' in graphs, a notion first introduced in \cite{milanicsep}. For $\lambda\in \poi$, we say that a graph $G$ is \textit{$\lambda$-separable} if for every two nonadjacent vertices $x,y$ in $G$, there is no set of $\lambda$ pairwise internally disjoint paths in $G$ from $x$ to $y$. 

We need an extension of separability that involves more than two vertices. For a graph $G$ and $\kappa,\lambda\in \poi$, a \textit{$(\kappa,\lambda)$-block in $G$} is a pair $(B,\mca{P})$ where $B$ is a set of at least $\kappa$ vertices in $G$, and $\mca{P}$ is a function with domain $\binom{B}{2}$ that maps each $2$-subset $\{x,y\}$ of $B$ to a set $\mca{P}(\{x,y\})=\mca{P}_{\{x,y\}}$ of at least $\lambda$ pairwise internally disjoint paths in $G$ from $x$ to $y$. We say that $(B,\mca{P})$ is \textit{strong} if, for all distinct $2$-subsets $\{x,y\}$ and $\{x',y'\}$ of $B$, we have $V(\mca{P}_{\{x,y\}})\cap V(\mca{P}_{\{x',y'\}})=\{x,y\}\cap \{x',y'\}$. We say that $(B,\mca{P})$ is \textit{stable} if $B$ is a stable set in $G$. We say that $G$ is \textit{$(\kappa,\lambda)$-separable} if there is no stable strong $(\kappa,\lambda)$-block in $G$. Note that $G$ is $\lambda$-separable if and only if $G$ is $(2,\lambda)$-separable.

\begin{restatable}{theorem}{mainpolyblock}\label{thm:mainpolyblock}
    For every planar graph $H$ and every $\sigma\in \poi$, there is a constant $d_{\ref{thm:mainpolyblock}}=d_{\ref{thm:mainpolyblock}}(H,\sigma)\in \poi$ such that for all $\kappa,\lambda,t\in \poi$, if $G$ is a $(\kappa,\lambda)$-separable $K_t$-free graph with no induced minor isomorphic to $H$ or $K_{\sigma,\sigma}$, then $\tw(G)\leq (\kappa\lambda t)^{d_{\ref{thm:mainpolyblock}}}$.
\end{restatable}

\begin{restatable}{theorem}{mainpolysep}\label{thm:mainpolysep}
    For all $\rho,\sigma\in \poi$, there are constants $c_{\ref{thm:mainpolysep}}=c_{\ref{thm:mainpolysep}}(\sigma)\in \poi$ and $d_{\ref{thm:mainpolysep}}=d_{\ref{thm:mainpolysep}}(\rho,\sigma)\in \poi\cup\{0\}$ such that for all $t\in \poi$, every $K_t$-free graph with no induced minor isomorphic to $\bin_{\rho}$ or $K_{\sigma,\sigma}$ is $(c_{\ref{thm:mainpolysep}}, t^{d_{\ref{thm:mainpolysep}}})$-separable.
\end{restatable}

From Theorems~\ref{thm:pwtotw}, \ref{thm:mainpolyblock}, and \ref{thm:mainpolysep}, we deduce \Cref{thm:indms} below. This improves the main result of \cite{tw18} by Chudnovsky, Spirkl, and the author, which states that under the same assumptions, the pathwidth of $G$ is bounded by \textsl{some} function of $\rho,\sigma$, and $t$, which is far from polynomial in $t$. (Similar to our approach in \cite{tw18}, one may also combine \Cref{thm:indms} with the main result of \cite{tw16} to characterize the induced subgraph obstructions to bounded pathwidth, \textsl{with} a polynomial bound in the clique number. We omit the details.)

\begin{theorem}\label{thm:indms}
    For all $\rho,\sigma\in \poi$, there is a constant $d_{\ref{thm:indms}}=d_{\ref{thm:indms}}(\rho,\sigma)\in \poi$ such that for every $t\in \poi$, every $K_{t}$-free graph $G$ with no induced minor isomorphic to $\bin_{\rho}$ or $K_{\sigma,\sigma}$ satisfies $\pw(G)<t^{d_{\ref{thm:indms}}}$.
\end{theorem}

\begin{proof}[Proof {\sl (assuming \ref{thm:pwtotw}, \ref{thm:mainpolyblock}, and \ref{thm:mainpolysep})}]
    Let $d=d_{\ref{thm:mainpolyblock}}(\bin_{\rho},\sigma)$; note that $d$ exists because $\bin_{\rho}$ is planar. Let $c'=c_{\ref{thm:mainpolysep}}(\sigma)$ and let $d'=d_{\ref{thm:mainpolysep}}(\rho,\sigma)$. Let
$$d_{\ref{thm:indms}}=d_{\ref{thm:indms}}(\rho,\sigma)=(6\rho+2)(c'+d'+1)d.$$

Let $G$ be a $K_{t}$-free graph with no induced minor isomorphic to $\bin_{\rho}$ or $K_{\sigma,\sigma}$. If $t=1$, then $\pw(G)=0<t^{d_{\ref{thm:indms}}}$. Assume that $t\geq 2$. Thus, by \Cref{thm:mainpolysep}, $G$ is $(c', t^{d'})$-separable. By \Cref{thm:mainpolyblock} (and since $c'<2^{c'}\leq t^{c'}$), we have $\tw(G)\leq (c't^{d'}\cdot t)^{d}<t^{(c'+d'+1) d}$. Hence, by \Cref{thm:pwtotw}, we have $\pw(G)<(t^{(c'+d'+1)d})^{(6\rho+2)}=t^{d_{\ref{thm:indms}}}$.
\end{proof}

We derive \Cref{thm:mainjagged} from \Cref{thm:indms} combined with some preexisting results. First, we need the next two results from our paper \cite{tw16} with Chudnovsky and Spirkl.

\begin{theorem}[Chudnovsky, Hajebi, Spirkl {\cite[Theorem~1.2]{tw16}}]\label{thm:compbip}
   For all $p,q,r\in \poi$, there is a constant $c_{\ref{thm:compbip}}=c_{\ref{thm:compbip}}(p,q,r)\in \poi$ such that for every graph $G$ with an induced minor isomorphic to $K_{c_{\ref{thm:compbip}},c_{\ref{thm:compbip}}}$, one of the following holds: 
\begin{enumerate}[{\rm (a)}]
  \item\label{thm:compbip_a} $G$ has an induced minor isomorphic to the $r$-wall.
  \item\label{thm:compbip_b} There is a $(p,q)$-constellation in $G$.
\end{enumerate} 
\end{theorem}

Let $G$ be a graph and let $\mf{c}=(A_1,\ldots, A_p;B_1,\ldots, B_q)$ be a $(p,q)$-constellation in $G$ where $p,q\in \poi$. For $d\in \poi$, we say that $\mf{c}$ is \textit{$d$-ample} if for all distinct $i,j\in \{1,\ldots, p\}$, say $V(A_i)=\{x_i\}$ and $V(A_j)=\{x_j\}$, and every $k\in \{1,\ldots, q\}$, every path $R$ in $G$ from $x_i$ to $x_j$ with $V(R^*)\subseteq V(B_k)$ has length at least $d+2$. In particular, $\mf{c}$ is $1$-ample if and only if no two vertices in $V(A_1)\cup \cdots\cup V(A_p)$ have common neighbors in $V(B_1)\cup \cdots \cup V(B_q)$.

\begin{lemma}[Chudnovsky, Hajebi, Spirkl {\cite[Lemma~4.1]{tw16}}]\label{lem:compbip}
   For all $d,p,q,r,s\in \poi$, there is a constant $c_{\ref{lem:compbip}}=c_{\ref{lem:compbip}}(d,p,q,r,s)\in \poi$ such that for every graph $G$, if there is a $(c_{\ref{lem:compbip}},c_{\ref{lem:compbip}})$-constellation in $G$, then one of the following holds: 
\begin{enumerate}[{\rm (a)}]
  \item\label{lem:compbip_a} $G$ has an induced subgraph isomorphic to either $K_{r,r}$ or a proper subdivision of $K_{s}$.
  \item\label{lem:compbip_b} There is a $d$-ample $(p,q)$-constellation in $G$.
\end{enumerate} 
\end{lemma}

We also use the well-known fact that every planar graph is an induced minor of a sufficiently large square grid. For instance, it is proved in \cite{twhad} that:

\begin{theorem}[Campbell, Davies, Distel, Frederickson, Gollin, Hendrey, Hickingbotham, Wiederrecht, Wood, Yepremyan {\cite[Theorem~12]{twhad}}]\label{thm:gridtograph}
    Every planar graph $H$ on $h\in \poi$ vertices is isomorphic to an induced minor of the $28h$-grid.
\end{theorem}

We deduce that:

\begin{corollary}\label{cor:compbip}
   For all $d,h,p,q,r\in \poi$, there is a constant $c_{\ref{cor:compbip}}=c_{\ref{cor:compbip}}(d,h,p,q,r)\in \poi$ such that for every graph $G$ with an induced minor isomorphic to $K_{c_{\ref{cor:compbip}},c_{\ref{cor:compbip}}}$, one of the following holds: 
\begin{enumerate}[{\rm (a)}]
  \item\label{cor:compbip_a} $G$ has an induced subgraph isomorphic to $K_{r,r}$.
  \item\label{cor:compbip_b} For every planar graph $H$ on $h$ vertices, $G$ has an induced minor isomorphic to $H$.
  \item\label{cor:compbip_c} There is a $d$-ample $(p,q)$-constellation in $G$.
\end{enumerate}
\end{corollary}
\begin{proof}
  \sloppy Let $\sigma=c_{\ref{lem:compbip}}(d,p,q,r,h)$. Let $$c_{\ref{cor:compbip}}=c_{\ref{cor:compbip}}(d,h,p,q,r)=c_{\ref{thm:compbip}}(\sigma,\sigma, 56h-1).$$
  Let $G$ be a graph with an induced minor isomorphic to $K_{c_{\ref{cor:compbip}},c_{\ref{cor:compbip}}}$. By \Cref{thm:compbip}, either $G$ has an induced minor isomorphic to the $(56h-1)$-wall, or there is a $(\sigma,\sigma)$-constellation in $G$. In the former case, by \Cref{obs:wall}, $G$ has an induced minor isomorphic to the $28h$-grid. By \Cref{thm:gridtograph}, $G$ has induced minors isomorphic to every $h$-vertex planar graph, and so \ref{cor:compbip}\ref{cor:compbip_b} holds. In the latter case, by \Cref{lem:compbip}, one of the following holds:
\begin{itemize}
    \item $G$ has an induced subgraph isomorphic to $K_{r,r}$.
    \item $G$ has an induced subgraph isomorphic to a proper subdivision of $K_h$.
    \item There is a $d$-ample $(p,q)$-constellation in $G$.
\end{itemize}
The first and third bullets are the same as \ref{cor:compbip}\ref{cor:compbip_a} and \ref{cor:compbip}\ref{cor:compbip_c}, and the second bullet implies \ref{cor:compbip}\ref{cor:compbip_b}. This completes the proof of \Cref{cor:compbip}.
\end{proof}

Finally, we need the following:

\begin{lemma}[Hickingbotham \cite{completeminorpw}]\label{lem:binaryisg}
 For every $r\in \poi$, if a graph $G$ has an \mbox{induced} \mbox{minor} isomorphic to $\bin_{8r}$, then $G$ has an induced subgraph isomorphic to either a \mbox{subdivision} of $\bin_{r}$ or the line graph of a subdivision of $\bin_{r}$.
\end{lemma}

We are now in a position to prove \Cref{thm:mainjagged}, which we restate:

\mainjagged*
\begin{proof} 
By assumption, for some $r\in \poi$, we have $K_{r,r}\notin \mca{G}$ (thus every graph in $\mca{G}$ is $K_{r,r}$-free), and for some $\zeta,\xi\in \poi$, every $\zeta$-jagged graph $J\in \mca{G}$ satisfies $\pw(J)<\xi$. 

We claim that:

\sta{\label{st:noisgbinary} For every $G\in \mca{G}$, there is no induced subgraph of $G$ isomorphic to a subdivision of $\bin_{2(\zeta+2)\xi}$ or to the line graph of a subdivision of $\bin_{2(\zeta+2)\xi}$.}

For suppose some graph $G\in \mca{G}$ has an induced subgraph isomorphic to a subdivision of $\bin_{2(\zeta+2)\xi}$ or to the line graph of a subdivision of $\bin_{2(\zeta+2)\xi}$. Note that $\bin_{2(\zeta+2)\xi}$ has an induced subgraph isomorphic to the $(\zeta+1)$-subdivision of $\bin_{2\xi}$. It follows that there is an induced subgraph $J$ of $G$ such that $J$ is isomorphic to either a $(\geq \zeta+1)$-subdivision of $\bin_{2\xi}$ or the line graph of a $(\geq \zeta+1)$-subdivision of $\bin_{2\xi}$. In particular, $J\in \mca{G}$ (because $J$ is an induced subgraph of $G\in \mca{G}$ and $\mca{G}$ is hereditary), and $\pw(J)\geq \xi$. Thus $J$ is not $\zeta$-jagged. On the other hand, since all subdivisions of $K_{1,3}$ and their line graphs have pathwidth at most two, it follows that for every induced subgraph $H$ of $J$ with $\pw(H)\ge 3$, there is a component $C$ of $H$ that is not isomorphic to an induced subgraph of a subdivision of $K_{1,3}$ or to the line graph of a subdivision of $K_{1,3}$. Consequently, there are nonadjacent vertices $u,v$ in $C$ such that both $u$ and $v$ have degree three in $J$, and so the unique path $P$ in $J$ from $u$ to $v$ is also a path in $C$. It follows that $V(P^*)$ is a set of at least $\zeta$ vertices of degree two in $C$, and so in $H$. But now $J$ is $\zeta$-jagged, a contradiction. This proves \eqref{st:noisgbinary}.
\vspace{2mm}

Let $\rho=16(\zeta+2)\xi$. From \eqref{st:noisgbinary}, \Cref{lem:binaryisg}, and the choice of $\rho$, it follows that:

\sta{\label{st:nobinary} For every $G\in \mca{G}$, there is no induced minor of $G$ isomorphic to $\bin_{\rho}$.}

Let $\sigma=c_{\ref{cor:compbip}}(\zeta+1,2^{\rho+1}-1,\xi,\xi,r)$. Notice that $\rho$ and $\sigma$ depend only on $r,\zeta,\xi$, which in turn depend only on the class $\mca{G}$. We claim that:

\sta{\label{st:nocompbip} For every $G\in \mca{G}$, there is no induced minor of $G$ isomorphic to $K_{\sigma,\sigma}$.}

Suppose for a contradiction that some graph $G\in \mca{G}$ has an induced minor isomorphic to $K_{\sigma,\sigma}$. By the choice of $\sigma$, we may apply \Cref{cor:compbip}. Recall that $G$ is $K_{r,r}$-free, and by \eqref{st:nobinary}, $G$ has no induced minor isomorphic to the $(2^{\rho+1}-1)$-vertex planar graph $\bin_{\rho}$. Thus \ref{cor:compbip}\ref{cor:compbip_c} holds: there is a $(\zeta+1)$-ample $(\xi,\xi)$-constellation $\mf{c}=(A_1,\ldots, A_{\xi};B_{1},\ldots, B_{\xi})$ in $G$. Let $A=V(A_1)\cup \cdots \cup V(A_{\xi})$, let $B=V(B_1)\cup \cdots \cup V(B_{\xi})$, and let $J=G[A\cup B]$. Then $J\in \mca{G}$ (because $J$ is an induced subgraph of $G\in \mca{G}$ and $\mca{G}$ is hereditary) with $\pw(J)\geq\xi$, and so $J$ is not $\zeta$-jagged. Let $H$ be an induced subgraph of $J$. Assume first that $|V(C)\cap A|\leq 1$ for every component $C$ of $H$. Then, since every component of $H[V(H)\cap B]$ is a path, it follows that $\pw(H)\leq 2$. Assume now that $|V(C)\cap A|\geq 2$ for some component $C$ of $H$. Choose distinct vertices $x,x'\in V(C)\cap A$ with $\dist_C(x,x')$ as small as possible. Let $R$ be a shortest path in $C$ from $x$ to $x'$, and let $y,y'$ be the neighbors of $x,x'$ in $R$; thus $y,y'\in B$. Since $\mf{c}$ is $(\zeta+1)$-ample, it follows that $x$ is the only neighbor of $y$ in $A$ and $x'$ is the only neighbor of $y'$ in $A$; in particular, $y$ and $y'$ are distinct, and $R$ has length at least three. So by the choice of $x$ and $x'$, we have that $(V(C)\cap A)\setminus \{x,x'\}$ and $V(R^*)$ are anticomplete in $C$. Therefore, there exists $k\in \{1,\ldots, \xi\}$ such that $V(R^*)\subseteq V(B_k)$, and so $|V(R^*)\setminus \{y,y'\}|\geq \zeta$ because $\mf{c}$ is $(\zeta+1)$-ample. Moreover, since $(V(C)\cap A)\setminus \{x,x'\}$ and $V(R^*)$ are anticomplete in $C$, and since $R$ is a path in $C$, it follows that $V(C)\cap A$ and $V(R^*)\setminus \{y,y'\}$ are anticomplete in $C$. Thus, every vertex in $V(R^*)\setminus \{y,y'\}$ has degree two in $C$, and so in $H$. In conclusion, we have shown that either $\pw(H)\leq 2$, or $H$ has at least $\zeta$ vertices of degree two. But now $J$ is $\zeta$-jagged, a contradiction. This proves \eqref{st:nocompbip}.
\vspace{2mm}

By \eqref{st:nobinary}, \eqref{st:nocompbip}, and \Cref{thm:indms}, for every $t\in \poi$, every $K_{t}$-free graph $G\in \mca{G}$ satisfies $\pw(G)< t^{d_{\ref{thm:indms}}(\rho,\sigma)}$. Hence, every graph $G\in \mca{G}$ satisfies $\pw(G)< (\omega(G)+1)^{d_{\ref{thm:indms}}(\rho,\sigma)}$, and so $\mca{G}$ is polynomially $(\pw,\omega)$-bounded. This completes the proof of \Cref{thm:mainjagged}.
\end{proof}
It remains to prove Theorems~\ref{thm:pwtotw}, \ref{thm:mainpolyblock}, and \ref{thm:mainpolysep}, which we do in Sections~\ref{sec:pwtotw}, \ref{sec:mainpolyblock}, and \ref{sec:mainpolysep}.

\section{Pathwidth versus treewidth}\label{sec:pwtotw}

Our goal here is to prove \Cref{thm:pwtotw}. This will be done in several steps, and we begin with some definitions.

A \textit{rooted tree} is a pair $(T,u)$ such that $T$ is a tree and $u\in V(T)$. Let $(T,u)$ be a rooted tree and let $v\in V(T)$. For each $i\in \{0,\ldots, \dist_T(u,v)\}$, the \textit{$i$-ancestor of $v$ in $(T,u)$} is the unique vertex $v'$ of the unique path in $T$ from $u$ to $v$ for which $\dist_T(v',v)=i$. It follows that $u$ is the $\dist_T(u,v)$-ancestor of $v$, and $v$ is the $0$-ancestor of $v$. By an \textit{ancestor of $v$ in $(T,u)$} we mean the $i$-ancestor of $v$ in $(T,u)$ for some $i\in \{0,\ldots, \dist_T(u,v)\}$. A \textit{child of $v$ in $(T,u)$} is a vertex in $V(T)\setminus \{v\}$ of which $v$ is the $1$-ancestor. Given a rooted tree $(T,u)$, a \textit{rooted subtree of $(T,u)$} is a rooted tree $(T',u')$ such that $T'$ is an induced subgraph of $T$, and for all $v\in V(T')$, every child of $v$ in $(T',u')$ is also a child of $v$ in $(T,u)$. 

For $\delta\in \poi$ and $\rho\in \poi\cup \{0\}$, we say that a rooted tree $(T,u)$ is \textit{$(\delta,\rho)$-regular} if every vertex of $T$ is at distance at most $\rho$ from $u$ in $T$, and every vertex in $V(T)\setminus N^{\rho}_T(u)$ has exactly $\delta$ children in $(T,u)$. 
\medskip

First, we need a Ramsey-type lemma:

\begin{lemma}\label{lem:coloruniform}
   Let $\delta,\gamma\in \poi$, let $\rho\in \poi\cup \{0\}$ and let $(T,u)$ be a $(\delta\gamma,\rho)$-regular rooted tree. Let $I$ be a non-empty set with $|I|\leq \gamma$ and let $\Phi:N^{\rho}_T(u)\rightarrow I$ be a function. Then there exists $i\in I$ as well as a $(\delta,\rho)$-regular rooted subtree $(T',u)$ of $(T,u)$ -- in particular, we have $N^{\rho}_{T'}(u)\subseteq N^{\rho}_T(u)$ -- such that $\Phi(v)=i$ for every $v\in N^{\rho}_{T'}(u)$.
\end{lemma}
\begin{proof}
We proceed by induction on $\rho$ (for fixed $\delta,\gamma$). If $\rho=0$, then $V(T)=\{u\}$, and the result is immediate by choosing $T'=T$ and $i=\Phi(u)$. 

Assume that $\rho\geq 1$. For each $x\in N_T^{\rho-1}(u)$, let $C_x$ be the set of all children of $x$ in $(T,u)$. Then $(C_x:x\in N_T^{\rho-1}(u))$ are pairwise disjoint $\delta\gamma$-subsets of $V(T)$ with $\bigcup_{x\in N_T^{\rho-1}(u)}C_x=N^{\rho}_{T}(u)\subseteq V(T)\setminus \{u\}$.

Let $T^-=T\setminus N^{\rho}_{T}(u)$. Then $u\in V(T^-)$ and $(T^-,u)$ is a $(\delta\gamma,\rho-1)$-regular rooted subtree of $(T,u)$; in particular, we have $N_{T^-}^{\rho-1}(u)=N_T^{\rho-1}(u)$. For each $x\in N_{T^-}^{\rho-1}(u)=N_T^{\rho-1}(u)$, since $C_x\subseteq N^{\rho}_{T}(u)$ with $|C_x|=\delta\gamma$, and since $|I|\leq \gamma$, it follows that there exists $i_x\in I$ as well as a $\delta$-subset $D_x$ of $C_x$ such that 
$$\Phi(v)=i_x$$
for every $v\in D_x$. 

Define the function $\Phi^-:N^{\rho-1}_{T^-}(u)\rightarrow I$ as 
$$\Phi^-(x)=i_x$$
for every $x\in N^{\rho-1}_{T^-}(u)$. By the inductive hypothesis applied to $(T^-,u)$ and $\Phi^-$, there exists $i\in I$ as well as a $(\delta,\rho-1)$-regular rooted subtree $(T'',u)$ of $(T^-,u)$ (and thus $N^{\rho-1}_{T''}(u)\subseteq N^{\rho-1}_{T^-}(u)$), such that
$$\Phi^-(x)=i$$
for every $x\in N^{\rho-1}_{T''}(u)$.

Now, let 
   $$T'=T\left[V(T'')\cup \left(\bigcup_{x\in N^{\rho-1}_{T''}(u)}D_x\right)\right].$$
   Since $(T'',u)$ is a $(\delta,\rho-1)$-regular rooted subtree of $(T^-,u)$, $(T^-,u)$ is a $(\delta\gamma,\rho-1)$-regular rooted subtree of $(T,u)$, and since $|D_x|=\delta$ for all $x\in N^{\rho-1}_{T''}(u)\subseteq N^{\rho-1}_{T^-}(u)$, it follows that $(T',u)$ is a $(\delta,\rho)$-regular rooted subtree of $(T,u)$, where
   $$N^{\rho}_{T'}(u)=\bigcup_{x\in N^{\rho-1}_{T''}(u)}D_x.$$
In particular, for every $v\in N^{\rho}_{T'}(u)$, the $1$-ancestor of $v$ in $(T',u)$ is the unique vertex $x\in N^{\rho-1}_{T''}(u)\subseteq N^{\rho-1}_{T^-}(u)$ for which $v\in D_x$, and thus $\Phi(v)=i_x=\Phi^-(x)=i$. This completes the proof of \Cref{lem:coloruniform}.
\end{proof}

Let $G$ be a graph, let $\delta\in \poi$, let $\rho\in \poi\cup \{0\}$ and let $(T,u)$ be a $(\delta,\rho)$-regular rooted tree where $T$ is a subgraph of $G$. 

\begin{itemize}
     \item We say that $(T,u)$ is \textit{path-uniform in $G$} if for all $v,v'\in N^{\rho}_{T}(u)$ and $i,j\in \{0,\ldots, \rho\}$, the $i$-ancestor and the $j$-ancestor of $v$ in $(T,u)$ are adjacent in $G$ if and only if the $i$-ancestor and the $j$-ancestor of $v'$ in $(T,u)$ are adjacent in $G$.
     \item We say that $(T,u)$ is \textit{path-induced in $G$} if for all $v\in N^{\rho}_{T}(u)$ and $i,j\in \{0,\ldots, \rho\}$, the $i$-ancestor and the $j$-ancestor of $v$ in $(T,u)$ are adjacent in $G$ if and only if $|i-j|=1$; that is, the unique path in $T$ from $u$ to $v$ is a path in $G$ (recall that by a ``path in $G$'' we mean an \textsl{induced} subgraph of $G$ that is a path).
      \item We say that $(T,u)$ is \textit{branch-induced in $G$} if for every edge $uv\in E(G)\setminus E(T)$ with $u,v\in V(T)$, one of $u,v$ is an ancestor of the other.
\end{itemize}

Observe that $T$ is an induced subgraph of $G$ if and only if $(T,u)$ is both branch-induced and path-induced in $G$.
\medskip

Using \Cref{lem:coloruniform}, we show that:

\begin{lemma}\label{lem:pathuniform}
   Let $\delta\in \poi$ and let $\rho\in \poi\cup \{0\}$. Let $G$ be a graph and let $(T,u)$ be a $(\delta 2^{\rho^2},\rho)$-regular rooted tree where $T$ is a subgraph of $G$. Then there is a $(\delta,\rho)$-regular rooted subtree $(T',u)$ of $(T,u)$ that is path-uniform in $G$.
\end{lemma}
\begin{proof}
Let 
$\displaystyle \mca{I}=2^{\binom{\{0,\ldots, \rho\}}{2}}$; that is, $\mca{I}$ is the set of all subsets of the set of all $2$-subsets of $\{0,\ldots, \rho\}$. For each $v\in N^{\rho}_{T}(u)$, let $I_v\in \mca{I}$ be the set of those $2$-subsets $\{i,j\}$ of $\{0,\ldots, \rho\}$ for which the $i$-ancestor and the $j$-ancestor of $v$ in $(T,u)$ are adjacent in $G$. 

Define the function $\Phi:N^{\rho}_{T}(u)\rightarrow \mca{I}$ as 
$$\Phi(v)=I_v$$
for every $v\in N^{\rho}_{T}(u)$. Since $(T,u)$ is $(\delta 2^{\rho^2}, \rho)$-regular and since $0<|\mca{I}|\leq 2^{\rho^2}$, by \Cref{lem:coloruniform}, there exists $I\in \mca{I}$ as well as a $(\delta,\rho)$-regular rooted subtree $(T',u)$ of $(T,u)$ (and thus $N^{\rho}_{T'}(u)\subseteq N^{\rho}_{T}(u)$), such that 
$$\Phi(v)=I$$ 
for every $v\in N^{\rho}_{T'}(u)$.

Now, for all $v,v'\in N^{\rho}_{T'}(u)$, we have
$$I_v=\Phi(v)=I=\Phi(v')=I_{v'}.$$
It follows that for every choice of $i,j\in \{0,\ldots, \rho\}$, the $i$-ancestor and the $j$-ancestor of $v$ in $(T,u)$ are adjacent in $G$ if and only if $\{i,j\}\in I$, if and only if the $i$-ancestor and the $j$-ancestor of $v'$ in $(T,u)$ are adjacent in $G$. Hence, by definition, $(T',u)$ is path-uniform in $G$. This completes the proof of \Cref{lem:pathuniform}.
\end{proof}

We need another Ramsey-type lemma:

\begin{lemma}\label{lem:longpath}
   Let $\rho,\tau\in \poi$, let $G$ be a graph and let $P=x_0\dd x_1\dd \cdots\dd x_{\tau^{\rho}}$ be a path of length $\tau^{\rho}$ that is a (not necessarily induced) subgraph of $G$. Then one of the following holds.
   \begin{enumerate}[{\rm (a)}]
       \item\label{lem:longpath_a} There exists $i_0\in \{0,\ldots, \tau^{\rho}\}$ such that $x_{i_0}$ is adjacent in $G$ to at least $\tau$ vertices in $\{x_i:i_0+1\leq i\leq \tau^{\rho}\}$.
       \item\label{lem:longpath_b} There are $j_0,j_1,\ldots,j_{\rho}\in \{0,\ldots, \tau^{\rho}\}$ with $0=j_0<j_1<\cdots<j_{\rho}$ such that $$E(G[\{x_{j_0},x_{j_1},\cdots, x_{j_{\rho}}\}])=\{x_{j_{k-1}}x_{j_{k}}:1\leq k\leq \rho\}.$$
       In other words, $x_{0}\dd x_{j_1}\dd \cdots \dd x_{j_{\rho}}$ is a path in $G$.
   \end{enumerate}
\end{lemma}
\begin{proof}
    We proceed by induction on $\rho$ (for fixed $\tau$). If $\rho=1$, then \ref{lem:longpath}\ref{lem:longpath_b} holds trivially (with $j_1=1$). So we may assume that $\rho\geq 2$.
    
    Let $\tau'=|N_G(x_0)\cap \{x_1,\ldots, x_{\tau^{\rho}}\}|$. Then $\tau'\geq 1$ because $x_0x_1\in E(G)$. In particular, there are $i_1,\ldots,i_{\tau'}\in \{1,\ldots, \tau^{\rho}\}$ with $1=i_1<\cdots<i_{\tau'}$ such that $$N_G(x_0)\cap \{x_1,\ldots, x_{\tau^{\rho}}\}= \{x_{i_1},\ldots,x_{i_{\tau'}}\}.$$
    If $\tau'\geq \tau$, then \ref{lem:longpath}\ref{lem:longpath_a} holds (with $i_0=0$). Thus, we may assume that $\tau'\leq \tau-1$, and so $\tau\geq 2$.

    We claim that:
    
\sta{\label{st:longsegment} There exists $\ell\in \{1,\ldots, \tau'\}$ such that $i_{\ell}+\tau^{\rho-1}\leq \tau^{\rho}$, and $x_0$ is anticomplete to $\{x_i:i_{\ell}+1\leq i\leq i_{\ell}+\tau^{\rho-1}\}$ in $G$.}

Write $i_{\tau'+1}=\tau^{\rho}$. Then $1=i_1<\cdots<i_{\tau'}\leq i_{\tau'+1}$. Since $\rho\geq 2$ and $1\leq \tau'\leq \tau-1$, it follows that:
    \[\sum_{\ell=1}^{\tau'}(i_{\ell+1}-i_{\ell})=i_{\tau'+1}-i_{1}=\tau^{\rho}-1> \tau^{\rho-1}(\tau-1)\geq \tau^{\rho-1}\cdot \tau'\]
    and so $i_{\ell+1}-i_{\ell}>\tau^{\rho-1}$ for some $\ell\in \{1,\ldots, \tau'\}$. This proves \eqref{st:longsegment}.
    \medskip

   Henceforth, let $\ell$ be as given by \eqref{st:longsegment}. For each $i\in \{0,\ldots,\tau^{\rho-1}\}$, let $$y_{i}=x_{i_{\ell}+i}.$$
   
   Then $Q=y_0\dd y_1\dd\cdots\dd y_{\tau^{\rho-1}}$ is a path of length $\tau^{\rho-1}$ that is a subgraph of $G$. By the inductive hypothesis applied to $Q$, either 
   \begin{itemize}
       \item there exists $i'_0\in \{0,\ldots, \tau^{\rho-1}\}$ such that $y_{i'_0}$ is adjacent in $G$ to at least $\tau$ vertices in $\{y_{i}:i'_0+1\leq i\leq \tau^{\rho-1}\}$; or
       \item there are $j'_0,j'_1,\ldots,j'_{\rho-1}\in \{0,\ldots, \tau^{\rho-1}\}$ with $0=j'_0<j'_1<\cdots<j'_{\rho-1}$ such that $y_{j'_0}\dd y_{j'_1}\dd \cdots \dd y_{j'_{\rho-1}}$ is a path in $G$. 
   \end{itemize}
   
   In the former case, \ref{lem:longpath}\ref{lem:longpath_a} holds with $i_0=i_{\ell}+i'_0$. So, we may assume that the second bullet holds. For each $k\in \{1,\ldots, \rho\}$, let $j_{k}=i_{\ell}+j'_{k-1}$.
   Then $j_1,\ldots,j_{\rho}\in \{1,\ldots, \tau^{\rho}\}$ with $j_1<\cdots<j_{\rho}$. Also, for each $k\in \{1,\ldots, \rho\}$, we have $x_{j_k}=x_{i_{\ell}+j'_{k-1}}=y_{j'_{k-1}}$. It follows that $x_{j_1}\dd \cdots \dd x_{j_{\rho}}=y_{j'_0}\dd y_{j'_1}\dd \cdots \dd y_{j'_{\rho-1}}$ is a path in $G$. 
   
   Moreover, recall that $x_0$ is adjacent in $G$ to $x_{j_1}=y_{j'_0}=y_0=x_{i_{\ell}}$, and by \eqref{st:longsegment}, for each $k\in \{2,\ldots, \rho\}$, $x_0$ is not adjacent to $x_{j_k}=x_{i_{\ell}+j'_{k-1}}$ in $G$. Hence $x_0\dd x_{j_1}\dd \cdots \dd x_{j_{\rho}}$ is a path in $G$, and thus \ref{lem:longpath}\ref{lem:longpath_b} holds. This completes the proof of \Cref{lem:longpath}.
\end{proof}

The following lemma is the key to the proof of \Cref{thm:pwtotw}:

\begin{lemma}\label{lem:pathinduced}
Let $\delta,\rho,\tau\in \poi$ with $\delta\geq 2$, let $G$ be a graph with no subgraph isomorphic to $K_{\tau,\tau}$ and let $(T,u)$ be a $(\delta 2^{(2\tau)^{2\rho}},(2\tau)^{\rho})$-regular rooted tree where $T$ is a subgraph of $G$. Then there is a $(\delta,\rho)$-regular rooted tree $(T',u')$ where $T'$ is a subgraph of $G$ with $V(T')\subseteq V(T)$ and $(T',u')$ is path-induced~in~$G$. 
\end{lemma}

\begin{proof}
    By \Cref{lem:pathuniform}, there is a $(\delta,(2\tau)^{\rho})$-regular rooted subtree $(T'',u)$ of $(T,u)$ such that $(T'',u)$ is path-uniform in $G$. For each $v\in N^{(2\tau)^{\rho}}_{T''}(u)$ and every $i\in \{0,\ldots, (2\tau)^{\rho}\}$, let $x^v_i$ be the $i$-ancestor of $v$ in $(T'',u)$. It follows that
    $x^v_0\dd \cdots \dd x^{v}_{(2\tau)^{\rho}}$ is the unique path in $T$ from $v=x^v_0$ to $u=x^v_{(2\tau)^{\rho}}$.
    
    Henceforth, let $w\in N^{(2\tau)^{\rho}}_{T''}(u)$ be fixed. We claim that:

\sta{\label{st:smalldegreeforward} For every $i_0\in \{0,\ldots, (2\tau)^{\rho}\}$, $x^w_{i_0}$ is adjacent in $G$ to fewer than $2\tau$ vertices in $\{x^w_i:i_0+1\leq i\leq (2\tau)^{\rho}\}$.}

Suppose not. Then there are $i_0,i_1,\ldots, i_{\tau}\in \{0,\ldots, (2\tau)^{\rho}\}$ with $i_0+\tau<i_1<\cdots<i_{\tau}$ such that $x^w_{i_0}x^w_{i_j}\in E(G)$ for all $j\in \{1,\ldots, \tau\}$. Let $W$ be the set of all vertices $v\in N^{(2\tau)^{\rho}}_{T''}(u)$ for which $x^v_{i_1}=x^w_{i_1}$. Then $w\in W$, $|W|=\delta^{i_1}$, and for all $j\in \{1,\ldots, \tau\}$ and every $v\in W$, we have $x^v_{i_j}=x^w_{i_j}$. Let 
$X=\{x^v_{i_0}:v\in W\}$. Then $x^w_{i_0}\in X$, and since $i_1-i_0>\tau$ and $\delta\geq 2$, it follows that $|X|=\delta^{i_1-i_0}>2^{\tau}>\tau$. On the other hand, since $(T'',u)$ is path-uniform in $G$ and since $x^w_{i_0}\in X$ is adjacent to $x^w_{i_1},\ldots, x^w_{i_{\tau}}$ in $G$, it follows that every vertex in $X$ is adjacent in $G$ to every vertex among $x^w_{i_1},\ldots, x^w_{i_{\tau}}$. But then $G[X\cup \{x^w_{i_1},\ldots, x^w_{i_{\tau}}\}]$ has a subgraph isomorphic to $K_{\tau,\tau}$, a contradiction. This proves \eqref{st:smalldegreeforward}.
\medskip

 Next, we apply \Cref{lem:longpath} to $P=x^w_0\dd \cdots \dd x^{w}_{(2\tau)^{\rho}}$. From \eqref{st:smalldegreeforward}, it follows that \ref{lem:longpath}\ref{lem:longpath_b} holds; that is, there exist $j_0,j_1,\ldots,j_{\rho}\in \{0,\ldots, (2\tau)^{\rho}\}$ with $0=j_0<j_1<\cdots<j_{\rho}$ such that $x^w_{j_0}\dd x^w_{j_1}\dd \cdots \dd x^w_{j_{\rho}}$ is a path in $G$. This, along with the fact that $(T'',u)$ is path-uniform, implies that

\sta{\label{st:allinducedpaths} For every $v\in N^{(2\tau)^{\rho}}_{T''}(u)$, $$x^v_{j_0}\dd x^v_{j_1}\dd \cdots \dd x^v_{j_{\rho}}=v\dd x^v_{j_1}\dd \cdots \dd x^v_{j_{\rho}}$$
is a path in $G$.}

 Now, define pairwise disjoint subsets $X_0,\ldots, X_{\rho}$ of $V(T'')$ recursively, as follows: 
 \begin{itemize}
     \item Let $X_{\rho}=\{x^w_{j_{\rho}}\}$.
     \item Assume that $X_{i}$ is defined for some $i\in \{1,\ldots, \rho\}$. For each $x\in X_{i}$, choose a $\delta$-subset $Y_x$ of $N_{T''}^{j_{i}-j_{i-1}}(x)$; this is possible because $|N_{T''}^{j_{i}-j_{i-1}}(x)|=\delta^{j_{i}-j_{i-1}}\geq \delta$ (recall that $j_0=0$). Then, let
     $$X_{i-1}=\bigcup_{x\in X_{i}}Y_x.$$
 \end{itemize} 
 Note that $X_i\subseteq N^{(2\tau)^{\rho}-j_i}_{T''}(u)$ for every $i\in \{0,\ldots, \rho\}$.
 
 Let $u'=x^w_{j_{\rho}}$, and let $T'$ be the tree with $V(T')=X_0\cup\cdots\cup X_{\rho}\subseteq V(T)$ and
 $$E(T')=\bigcup_{i=1}^{\rho}\left\{xy:x\in X_i,\ y\in Y_x\right\}.$$
Then $(T',u')$ is a $(\delta,\rho)$-regular rooted tree. Moreover, by \eqref{st:allinducedpaths}, $T'$ is a subgraph of $G$ and $(T',u')$ is path-induced in $G$. This completes the proof of \Cref{lem:pathinduced}.    
\end{proof}

For our last lemma, we need Ramsey's theorem:

\begin{theorem}[Ramsey \cite{multiramsey}]\label{thm:ramsey}
  For all $\alpha,t\in \poi$, every $K_t$-free graph on at least $t^{\alpha-1}$ vertices has a stable set of cardinality $\alpha$.
\end{theorem}

\begin{lemma}\label{lem:branchinduced}
   Let $\delta,\tau\in \poi$ and let $\rho\in \poi\cup \{0\}$. Let $G$ be a graph with no minor isomorphic to $K_{\tau+1}$ and let $(T,u)$ be a $(\Delta,\rho)$-regular rooted tree, for some $\Delta\in \poi$ with $\Delta\geq \tau^{\delta-1}$, where $T$ is a subgraph of $G$. Then some $(\delta,\rho)$-regular rooted subtree $(T',u)$ of $(T,u)$ is branch-induced in $G$.
\end{lemma}

\begin{proof}
 We proceed by induction on $\rho$ (for fixed $\delta,\tau$). If $\rho=0$, then $V(T)=\{u\}$, and we are done by choosing $T'=T$. So we may assume that $\rho\geq 1$. 
 
 For every $v\in N_T(u)$ (that is, for every child $v$ of $u$ in $(T,u)$), let $T_v$ be the component of $T\setminus u$ containing $v$. Then $(T_v,v)$ is a $(\Delta,\rho-1)$-regular rooted subtree of $(T,u)$. Note that if there is a $\tau$-subset $\{v_1,\ldots, v_{\tau}\}$ of $N_T(u)$ such that no two of $V(T_{v_1}),\ldots, V(T_{v_{\tau}})$ are anticomplete in $G$, then $(G[\{u\}], G[V(T_{v_1})],\ldots, G[V(T_{v_{\tau}})])$ is a $K_{\tau+1}$-model in $G$, which is impossible. Therefore, since $|N_T(u)|=\Delta\geq \tau^{\delta-1}$, it follows from \Cref{thm:ramsey} that:
 
 \sta{\label{st:manyantisubtree} There is a $\delta$-subset $\{u_1,\ldots, u_{\delta}\}$ of $N_T(u)$ for which $V(T_{u_1}),\ldots, V(T_{u_{\delta}})$ are pairwise anticomplete in $G$.}
 
 On the other hand, for each $i\in \{1,\ldots, \delta\}$, by the inductive hypothesis applied to $(T_{u_i},u_i)$, there is a $(\delta,\rho-1)$-regular rooted subtree $(T'_i,u_i)$ of $(T_{u_i},u_i)$ such that $(T'_i,u_i)$ is branch-induced in $G$. Let $T'=T[\{u\}\cup V(T'_{1})\cup\cdots \cup V(T'_{\delta})]$.
Then, from \eqref{st:manyantisubtree}, it follows that $(T',u)$ is a $(\delta,\rho)$-regular rooted subtree of $(T,u)$ that is branch-induced in $G$. This completes the proof of \Cref{lem:branchinduced}.
\end{proof}

We also need a result from \cite{approxpw}, and an observation (whose proof is easy and we omit).

\begin{theorem}[Groenland, Joret, Nadara, Walczak {\cite[1.1]{approxpw}}]\label{thm:pwsub}
   For all $\rho,\tau\in \poi$, every graph $G$ with $\tw(G)<\tau$ and no minor isomorphic to $\bin_{\rho}$ satisfies $\pw(G)\leq (\rho-1)\tau+1$.
\end{theorem}

\begin{observation}\label{obs:trees}
   Let $\delta,\rho\in \poi$ and let $G$ be a graph with a minor isomorphic to $\bin_{\delta\rho}$. Then there is a $(2^{\delta},\rho)$-regular rooted tree $(T,u)$ such that $T$ is a minor of $G$.
\end{observation}

Finally, we are ready to prove \Cref{thm:pwtotw}, which we restate:

\pwtotw*

\begin{proof}
  Let $G$ be a graph with $\tw(G)<\tau$ and $\pw(G)\geq \tau^{6\rho+2}$. Then $\tau\geq 2$. Let $\rho_1=(2\tau)^{\rho}$.
Since $\tau\geq 2$, we have $\tau^{2\rho}=\tau^{\rho}\cdot\tau^{\rho} \geq 2^{\rho}\cdot\tau^{\rho}=\rho_1$, and so $\tau^{4\rho+1}\geq 2\tau^{4\rho}\geq \tau+\tau^{4\rho}\geq \tau+\rho_1^2$. Therefore,
$$\pw(G)\geq \tau^{6\rho+2}=\tau^{4\rho+1}\cdot \tau^{2\rho+1}\geq \left(\tau+\rho_1^2\right)\rho_1\tau> \left(\left(\tau+\rho_1^2\right)\rho_1-1\right)\tau+1.$$
Since $\tw(G)<\tau$, it follows from \Cref{thm:pwsub} that $G$ has a minor isomorphic to $\bin_{(\tau+\rho_1^2)\rho_1}$. This, combined with \Cref{obs:trees}, implies that there is an induced minor $G_1$ of $G$ as well as a $(2^{\tau}\cdot 2^{\rho_1^2},\rho_1)$-regular rooted tree $(T_1,u_1)$ such that $T_1$ is a subgraph of $G_1$.

Since $G_1$ is an induced minor of $G$, it follows that $\tw(G_1)\leq \tw(G)<\tau$, and so $G_1$ has no subgraph isomorphic to $K_{\tau,\tau}$. Consequently, since $\rho_1=(2\tau)^{\rho}$ and $2^{\tau}\geq 2$, it follows from \Cref{lem:pathinduced} applied to $G_1$ and $(T_1,u_1)$ (with $\delta=2^{\tau}$) that there is a $(2^{\tau},\rho)$-regular rooted tree $(T,u)$ with $V(T)\subseteq V(T_1)$ such that $T$ is a subgraph of $G_1$ and $(T,u)$ is path-induced in $G_1$. Since $\tw(G_1)<\tau$, it follows that $G_1$ has no minor isomorphic to $K_{\tau+1}$. Therefore, by \Cref{lem:branchinduced} applied to $G_1$ and $(T,u)$ (with $\Delta=2^{\tau}\geq\tau$ and $\delta=2$), there is a $(2,\rho)$-regular rooted subtree $(T',u)$ of $(T,u)$ that is branch-induced in $G_1$. Moreover, $(T',u)$ is path-induced in $G_1$ because $(T,u)$ is. It follows that $T'$ is an induced subgraph of $G_1$ isomorphic to $\bin_{\rho}$. Hence, $G$ has an induced minor isomorphic to $\bin_{\rho}$, because $T'$ is an induced subgraph of $G_1$, and $G_1$ is an induced minor of $G$. This completes the proof of \Cref{thm:pwtotw}.
\end{proof}

\section{Treewidth versus separability}\label{sec:mainpolyblock}

In this section we prove \Cref{thm:mainpolyblock}. The necessary lemmas appear in Subsections~\ref{subsec:edgecoloring}, \ref{subsec:indmtorso}, and \ref{subsec:subcompletesubg}, and we complete the proof of \ref{thm:mainpolyblock} in Subsection~\ref{subsec:finalblock}.

\subsection{Monochromatic components of bounded diameter}\label{subsec:edgecoloring}
Here we prove:

\begin{lemma}
    \label{lem:mainpolydegree}
    For all $h,\sigma\in \poi$, there is a constant $d_{\ref{lem:mainpolydegree}}=d_{\ref{lem:mainpolydegree}}(h,\sigma)\in \poi$ such that for every planar graph $H$ on $h$ vertices, every graph $G$ with maximum degree at most $\Delta\in \poi$ and no induced minor isomorphic to $H$ or $K_{\sigma,\sigma}$ satisfies $\tw(G)\leq  \Delta^{d_{\ref{lem:mainpolydegree}}}$.
\end{lemma}

This is a variant of \cite[Theorem~2]{BHKM} where the complete bipartite graph is excluded as a subgraph. The key tool in the proof is \Cref{lem:mainedgecluster} below.
Let $G$ be a graph. For $X\subseteq V(G)$, we write $\diam_G(X)=\max\{\dist_{G}(x,x'):x,x'\in X\}$. For $c\in \poi$, a \textit{$c$-edge-coloring of $G$} is a partition $(E_1,\ldots,E_{c})$ of $E(G)$. We say that $(E_1,\ldots, E_{c})$ has \textit{clustering less than $n\in \poi$ in $G$} if for each $i\in \{1,\ldots, c\}$, every component of the graph $(V(G),E_i)$ has fewer than $n$ vertices. We also say that $(E_1,\ldots,E_{c})$ has \textit{$G$-diameter less than $h\in \poi$} if for each $i\in \{1,\ldots,c\}$, every component $C$ of the graph $(V(G),E_i)$ satisfies $\diam_G(V(C))<h$.

\begin{lemma}\label{lem:mainedgecluster}
    For all $\sigma\in \poi$, there is a constant $c_{\ref{lem:mainedgecluster}}=c_{\ref{lem:mainedgecluster}}(\sigma)\in \poi$ such that every graph $G$ with no induced minor isomorphic to $K_{\sigma,\sigma}$ has a $c_{\ref{lem:mainedgecluster}}$-edge-coloring of $G$-diameter less than $12\sigma$.
\end{lemma}

To derive \Cref{lem:mainpolydegree} from \Cref{lem:mainedgecluster}, we need a result from \cite{BHKM}:

\begin{lemma}[Bonnet, Hodor, Korhonen, Masa\v{r}\'{i}k {\cite[Lemma~14]{BHKM}}]\label{lem:BHKM}
    There is an absolute constant $c_{\ref{lem:BHKM}}\in \poi$ with the following property. Let $H$ be a planar graph on $h\in \poi$ vertices and let $G$ be a graph with no induced minor isomorphic to $H$. Let $\gamma,\nu\in \poi$ such that $G$ admits a $\gamma$-edge-coloring of clustering less than $\nu$ in $G$. Then $\tw(G)\leq  h^{c_{\ref{lem:BHKM}}} 2^{c_{\ref{lem:BHKM}}(\gamma^5+\gamma \log \nu)}$.
\end{lemma}

We also need the following observation (the proof is straightforward and omitted):

\begin{observation}\label{obs:diamtocluster}
    Let $\Delta, \eta\in \poi$ with $\Delta\geq 2$, let $G$ be a graph with maximum degree at most $\Delta$ and let $X\subseteq V(G)$ with $\diam_G(X)<\eta$. Then $|X|<\Delta^{\eta}$.
\end{observation}

\Cref{lem:mainpolydegree} follows almost immediately from Lemmas~\ref{lem:mainedgecluster} and \ref{lem:BHKM} and \Cref{obs:diamtocluster}:

\begin{proof}[Proof of \Cref{lem:mainpolydegree}]
    Let $\gamma=c_{\ref{lem:mainedgecluster}}(\sigma)$ and let $\eta=12\sigma$; thus $\gamma$ and $\eta$ depend only on $\sigma$. Recall the absolute constant 
    $c_{\ref{lem:BHKM}}$ from \Cref{lem:BHKM}. Let
$$d_{\ref{lem:mainpolydegree}}=d_{\ref{lem:mainpolydegree}}(h,\sigma)=h^{c_{\ref{lem:BHKM}}}\cdot 2^{c_{\ref{lem:BHKM}}\gamma^5}+c_{\ref{lem:BHKM}}\gamma\eta.$$
If $\Delta=1$, then $\tw(G)\leq 1\leq \Delta^{d_{\ref{lem:mainpolydegree}}}$, and we are done. So we may assume that $\Delta\geq 2$. Since $G$ has no induced minor isomorphic to $K_{\sigma,\sigma}$, by \Cref{lem:mainedgecluster} and the choice of $\gamma,\eta$, there is a $\gamma$-edge-coloring of $G$ with $G$-diameter less than $\eta$. Since $G$ has maximum degree at most $\Delta\geq 2$, by \Cref{obs:diamtocluster}, the same coloring has clustering less than $\Delta^{\eta}$ in $G$. Since $G$ has no induced minor isomorphic to $H$, by \Cref{lem:BHKM}, we deduce that
$$\tw(G)\leq  h^{c_{\ref{lem:BHKM}}} 2^{c_{\ref{lem:BHKM}}(\gamma^5+\gamma \log \Delta^{\eta})}=h^{c_{\ref{lem:BHKM}}} 2^{c_{\ref{lem:BHKM}}\gamma^5}\Delta^{c_{\ref{lem:BHKM}}\gamma\eta}\leq 2^{h^{c_{\ref{lem:BHKM}}}\cdot 2^{c_{\ref{lem:BHKM}}\gamma^5}}\Delta^{c_{\ref{lem:BHKM}}\gamma\eta}\leq \Delta^{d_{\ref{lem:mainpolydegree}}},$$
where the last inequality follows from $\Delta\geq 2$. This completes the proof of \Cref{lem:mainpolydegree}.
\end{proof}

It remains to prove \Cref{lem:mainedgecluster}, which occupies the rest of this subsection.
\medskip

We start with a definition that is central to the proof. Let $\sigma,\theta\in \poi\cup \{0\}$, let $G$ be a graph and let $X\subseteq V(G)$. A \textit{$(\sigma,\theta)$-$X$-abyss in $G$} is a $(\sigma+2)$-tuple $(a_0,\ldots, a_{\sigma+1})$ of vertices in $G$ with the following specifications:
\begin{itemize}
    \item there is a connected induced subgraph $G_0$ of $G[X]$ such that $a_0,\ldots, a_{\sigma+1}\in V(G_0)$;
    \item for each $i\in \{0,\ldots, \sigma-1\}$, there exists $\rho_{i}\in \poi\cup \{0\}$ as well as a connected induced subgraph $G_{i+1}$ of $G[N_{G_i}^{\rho_{i}+1}(a_i)]$ such that $a_{i+1},\ldots, a_{\sigma+1}\in V(G_{i+1})$; and 
    \item we have $\dist_{G}(a_{\sigma},a_{\sigma+1})\geq \theta$.
\end{itemize}

We show that:

\begin{lemma}\label{lem:abyssmain}
 Let $\sigma\in \poi$, let $G$ be a graph with no induced minor isomorphic to $K_{\sigma,\sigma}$ and let $X\subseteq V(G)$. Then there is no $(\sigma,4\sigma)$-$X$-abyss in $G$.  
\end{lemma}

\begin{proof} 
Suppose for a contradiction that there is a $(\sigma,4\sigma)$-$X$-abyss $(a_0,\ldots, a_{\sigma+1})$ in $G$. Let $G_0$ and $(\rho_i,G_{i+1}:i\in \{0,\ldots, \sigma-1\})$ be as in the definition; that is, $G_0$ is a connected induced subgraph of $G[X]$ such that $a_0,\ldots, a_{\sigma+1}\in V(G_0)$, and for each $i\in \{0,\ldots, \sigma-1\}$, we have $\rho_i\in \poi\cup \{0\}$ and $G_{i+1}$ is a connected induced subgraph of $G[N_{G_i}^{\rho_{i}+1}(a_i)]$ such that $a_{i+1},\ldots, a_{\sigma+1}\in V(G_{i+1})$. In particular, we have $V(G_{\sigma})\subseteq \cdots\subseteq V(G_{0})$. Also, $G_{\sigma}$ is a connected induced subgraph of $G$ with $a_{\sigma}, a_{\sigma+1}\in V(G_{\sigma})$, where $a_{\sigma}, a_{\sigma+1}$ are distinct (because $\dist_{G}(a_{\sigma},a_{\sigma+1})\geq 4\sigma\geq 4$). It follows that $G$ has an induced minor isomorphic to $K_{1,1}$, and so $\sigma\geq 2$. Let us now show that:

\sta{\label{st:pwfarapart} There are $\sigma$ vertices $x_1,\ldots, x_{\sigma}\in V(G_{\sigma})$ that are pairwise at distance at least $4$ in $G$.}

Since $a_{\sigma},a_{\sigma+1}\in V(G_{\sigma})$ and $G_{\sigma}$ is connected, there is a path $P$ in $G_{\sigma}$ from $a_{\sigma}$ to $a_{\sigma+1}$. Note that $|\dist_{G}(a_{\sigma},x)-\dist_{G}(a_{\sigma},x')|\leq 1$ for every $xx'\in E(P)$. Since $\dist_{G}(a_{\sigma},a_{\sigma})=0$ and $\dist_{G}(a_{\sigma},a_{\sigma+1})\geq 4\sigma$, there exist $x_1,\ldots, x_{\sigma}\in V(P)\subseteq V(G_{\sigma})$ such that for each $i\in \{1,\ldots,\sigma\}$, we have $\dist_G(a_{\sigma},x_i)=4i$, and so for all $i,i'\in \{1,\ldots,\sigma\}$ with $i<i'$, we have $\dist_G(x_i,x_{i'})\geq \dist_G(a_{\sigma},x_{i'})-\dist_G(a_{\sigma},x_i)=4(i'-i)\geq 4$. This proves \eqref{st:pwfarapart}.
\medskip

Let $x_1,\ldots, x_{\sigma}\in V(G_{\sigma})$ be as in \eqref{st:pwfarapart}. For every $i\in \{1,\ldots,\sigma\}$ and $j\in \{0,\ldots,\sigma-1\}$, since $x_i\in V(G_{\sigma})\subseteq V(G_{j+1})\subseteq N_{G_j}^{\rho_{j}+1}(a_j)$, and since $G_{j}$ is connected, it follows that $x_i$ has a neighbor $y_{i,j}\in N_{G_j}^{\rho_{j}}(a_j)$ in $G_j$ (and so in $G$). For each $i\in \{1,\ldots,\sigma\}$, let 
$$A_i=G[\{x_i, y_{i,0},y_{i,1},\ldots,y_{i,\sigma-1}\}].$$

Then $A_1,\ldots, A_{\sigma}$ are connected, and by \eqref{st:pwfarapart}, $A_1,\ldots, A_{\sigma}$ are pairwise anticomplete in $G$. 

In particular, since $\sigma\geq 2$ and $A_1,\ldots, A_{\sigma}$ are pairwise anticomplete in $G$, it follows that for each $j\in \{0,\ldots,\sigma-1\}$, we have $\rho_j\geq 1$, which in turn implies that
$$B_j=G\left[\bigcup_{\rho=0}^{\rho_j-1}N^{\rho}_{G_{j}}(a_j)\right]$$
is non-null and connected. Moreover, for all $j,j'\in \{0,\ldots,\sigma-1\}$ with $j<j'$, by the definition of $B_j$ and since $V(B_{j'})\subseteq V(G_{j'})\subseteq V(G_{j+1})\subseteq N^{\rho_j+1}_{G_j}(a_j)$, it follows that $B_j$ and $B_{j'}$ are anticomplete in $G$. 

Furthermore, for all $i\in \{1,\ldots,\sigma\}$ and $j\in \{0,\ldots,\sigma-1\}$, it is immediate from the definitions of $A_i$ and $B_j$ that $V(A_i)\cap V(B_j)=\varnothing$. Also, since $G_j$ is connected and $y_{i,j}\in V(A_i)\cap N_{G_j}^{\rho_{j}}(a_j)$, it follows that $y_{i,j}\in V(A_i)$ has a neighbor in $N_{G_j}^{\rho_{j}-1}(a_j)\subseteq V(B_j)$; thus, $A_i$ and $B_j$ are not anticomplete in $G$. But now $(A_1,\ldots,A_{\sigma}; B_0,\ldots, B_{\sigma-1})$ is an induced $K_{\sigma,\sigma}$-model in $G$, a contradiction. This completes the proof of Lemma~\ref{lem:abyssmain}.
\end{proof}

We continue with more definitions. Let $G$ be a graph and let $c\in \poi$. A \textit{$c$-vertex-coloring of $G$} is a partition $(V_1,\ldots, V_{c})$ of $V(G)$. Let $H$ be an induced subgraph of $G$ and let $(V_1,\ldots, V_{c})$ be a $c$-vertex-coloring of $H$. We say that $(V_1,\ldots, V_{c})$ has \textit{$G$-diameter less than $h\in \poi$} if for each $i\in \{1,\ldots, c\}$, every component $C$ of $H[V_i]=G[V_i]$ satisfies $\diam_G(V(C))< h$. For instance, a $c$-vertex-coloring $(V_1,\ldots, V_{c})$ of a graph $G$ is \textit{proper} vertex-coloring if and only if $(V_1,\ldots, V_{c})$ has $G$-diameter less than $1$.

Next, we apply \Cref{lem:abyssmain} to prove that:

 \begin{lemma}\label{lem:mainvertexcluster}
      Let $\sigma\in \poi\cup \{0\}$, let $\theta\in \poi$, let $G$ be a graph and let $X\subseteq V(G)$ such that there is no $(\sigma,\theta)$-$X$-abyss in $G$. Then $G[X]$ admits a $2^{\sigma}$-vertex-coloring of $G$-diameter less than $\theta$.
 \end{lemma}
\begin{proof}
    We proceed by induction on $\sigma$ (for fixed $\theta$). It suffices to prove that every component $G_0$ of $G[X]$ admits a $2^{\sigma}$-vertex-coloring of $G$-diameter less than $\theta$. Assume that $\sigma=0$. Then there is no $(0,\theta)$-$X$-abyss in $G$, which in particular means that there are no two vertices $a_0,a_1\in V(G_0)$ with $\dist_{G}(a_0,a_1)\geq \theta$. But then $\diam_{G}(V(G_0))<\theta$, and $G_0$ admits a $1$-vertex-coloring of $G$-diameter less than $\theta$. Assume that $\sigma\geq 1$. Choose a vertex $a_0\in V(G_0)$. Since $G_0$ is connected, there exists $R\in \poi\cup \{0\}$ such that $N_{G_0}^{\rho}(a_0)\neq \varnothing$ for every $\rho\in \{0,\ldots, R\}$, and $N_{G_0}^{\rho}(a_0)=\varnothing$ for all $\rho\in \poi$ with $\rho>R$. We claim that:

    \sta{\label{st:colorlevels} For all $\rho\in \{0,\ldots, R\}$, $G[N_{G_0}^{\rho}(a_0)]$ admits a $2^{\sigma-1}$-vertex-coloring $\left(V^{\rho}_1,\ldots, V^{\rho}_{2^{\sigma-1}}\right)$ of $G$-diameter less than $\theta$.}

    Suppose not. Then by the inductive hypothesis, there is a $(\sigma-1,\theta)$-$N_{G_0}^{\rho}(a_0)$-abyss $(a_1,\ldots, a_{\sigma+1})$ in $G$. But now $(a_0,a_1,\ldots, a_{\sigma+1})$ is a $(\sigma,\theta)$-$X$-abyss in $G$, a contradiction. This proves \eqref{st:colorlevels}.
    \medskip    
    
   Henceforth, for each $\rho\in \{0,\ldots, R\}$, let $\left(V^{\rho}_1,\ldots, V^{\rho}_{2^{\sigma-1}}\right)$ be in \eqref{st:colorlevels}. Let $I_0$ be the set of all even numbers in $\{0,\ldots, R\}$ and let $I_1$ be the set of all odd numbers in $\{0,\ldots, R\}$. For each $i\in \{1,\ldots, 2^{\sigma-1}\}$, let
    $$V_{0,i}=\bigcup_{\rho\in I_0}V^{\rho}_i; \qquad V_{1,i}=\bigcup_{\rho\in I_1}V^{\rho}_i.$$
    
Then, from \eqref{st:colorlevels}, since $(N^{\rho}_{G_0}(a_0):\rho\in \{0,\ldots, R\})$ is a partition of $V(G_0)$, and since $N^{\rho}_{G_0}(a_0)$ and $N^{\rho'}_{G_0}(a_0)$ are anticomplete in $G_0$ for all $\rho,\rho'\in \{0,\ldots, R\}$ with $|\rho-\rho'|>1$, it follows that $(V_{0,i},V_{1,i}: i\in \{1,\ldots, 2^{\sigma-1}\})$ is a $2^{\sigma}$-vertex-coloring of $G_0$ with $G$-diameter less than $\theta$. This completes the proof of Lemma~\ref{lem:mainvertexcluster}.
\end{proof}

We need the next two results (but we do not need the definition of ``bounded expansion''). Recall that a \textit{star} is a graph isomorphic to $K_{1,s}$ for some $s\in \poi\cup \{0\}$. 

\begin{theorem}[Dvo\v{r}\'{a}k; {\cite[Theorem~4]{dvorak}}]\label{thm:dvorakexpansion}
    For every graph $H$ and every $\tau\in \poi$, the class of all graphs with no induced subgraph isomorphic to any subdivision of $H$ and no subgraph isomorphic to $K_{\tau,\tau}$ has ``bounded expansion''.
\end{theorem}

\begin{theorem}[Ne\v{s}et\v{r}il and Ossona de Mendez {\cite[Theorem~7.1]{NOdMexpansion}}]\label{thm:nodmexpansion}
    If a class $\mca{G}$ has ``bounded expansion'', then there is a constant $c=c(\mca{G})\in \poi$ such that every graph $G\in \mca{G}$ admits a proper $c$-vertex-coloring $(V_1,\ldots, V_{c})$ where for all distinct $i,j\in \{1,\ldots, c\}$, every component of $G[V_i\cup V_j]$ is a star.
\end{theorem}
(More precisely, \cite[Theorem~7.1]{NOdMexpansion} says that a number of conditions are equivalent for a class, and our \Cref{thm:nodmexpansion} is equivalent to ``(5) implies (3) for $p=2$'' therein.)
\medskip

Combining Theorems~\ref{thm:dvorakexpansion} and \ref{thm:nodmexpansion}, we deduce that:

\begin{corollary}\label{cor:starcomponents}
    For all $\tau\in \poi$ and every graph $H$, there is a constant $c_{\ref{cor:starcomponents}}=c_{\ref{cor:starcomponents}}(H,\tau)\in \poi$ such that every graph $G$ with no induced subgraph isomorphic to any subdivision of $H$ and no subgraph isomorphic to $K_{\tau,\tau}$ admits a proper $c_{\ref{cor:starcomponents}}$-vertex-coloring $(V_1,\ldots, V_{c_{\ref{cor:starcomponents}}})$ such that for all distinct $i,j\in \{1,\ldots, c_{\ref{cor:starcomponents}}\}$, every component of $G[V_i\cup V_j]$ is a star. 
\end{corollary}

Finally, we are ready to prove \Cref{lem:mainedgecluster}:

\begin{proof}[Proof of \Cref{lem:mainedgecluster}]
    Let $\tau=2^{\sigma(\sigma-1)}$ and let $\gamma=c_{\ref{cor:starcomponents}}\left(K_{\sigma,\sigma}, \tau\right)$; thus, $\gamma$ depends only on $\sigma$. We will show that 
$$c_{\ref{lem:mainedgecluster}}=c_{\ref{lem:mainedgecluster}}(\sigma)=\binom{\gamma}{2}$$
satisfies the theorem.

Let $G$ be a graph with no induced minor isomorphic to $K_{\sigma,\sigma}$. By \Cref{lem:abyssmain}, there is no $(\sigma,4\sigma)$-$V(G)$-abyss in $G$, and so by \Cref{lem:mainvertexcluster}, $G$ admits a $2^{\sigma}$-vertex-coloring $(V_1,\ldots, V_{2^{\sigma}})$ of $G$-diameter less than $4\sigma$. For each $i\in \{1,\ldots, 2^{\sigma}\}$, let $\mca{C}_i$ be the set of all components of $G[V_i]$. Let $\Gamma$ be the graph with vertex set $\mca{U}=\mca{C}_1\cup \cdots \cup \mca{C}_{2^{\sigma}}$ such that $C,C'\in \mca{U}$ are adjacent in $\Gamma$ if and only if $C,C'$ are not anticomplete in $G$. It follows that $G$ has an induced minor isomorphic to $\Gamma$, and so $\Gamma$ has no induced minor isomorphic to $K_{\sigma,\sigma}$; in particular, $\Gamma$ is $K_{\sigma,\sigma}$-free. Also, $(\mca{C}_1, \ldots, \mca{C}_{2^{\sigma}})$ is a proper $2^{\sigma}$-vertex-coloring of $\Gamma$, and so $\Gamma$ is $K_{2^{\sigma}+1}$-free.

By the choice of $\tau$ and since $\Gamma$ is $(K_{\sigma,\sigma}, K_{2^{\sigma}+1})$-free, it follows from \Cref{thm:ramsey} that $\Gamma$ has no subgraph isomorphic to $K_{\tau,\tau}$. Since $\Gamma$ has no induced minor isomorphic to $K_{\sigma,\sigma}$, it follows that $\Gamma$ has no induced subgraph isomorphic to any subdivision of $K_{\sigma,\sigma}$. Therefore, by \Cref{cor:starcomponents} and the choice of $\gamma$, there is a proper $\gamma$-vertex-coloring $(\mca{V}_1,\ldots, \mca{V}_{\gamma})$ of $\Gamma$ such that for all distinct $i,j\in \{1,\ldots, \gamma\}$, every component of $\Gamma[\mca{V}_i\cup \mca{V}_j]$ is a star. 

Recall that $c_{\ref{lem:mainedgecluster}}=\binom{\gamma}{2}$. Let $\{a_1,b_1\},\ldots, \{a_{c_{\ref{lem:mainedgecluster}}},b_{c_{\ref{lem:mainedgecluster}}}\}$ be an enumeration of all $2$-subsets of $\{1,\ldots, \gamma\}$. For each $i\in \{1, \ldots, c_{\ref{lem:mainedgecluster}}\}$, let
$$U_i=\bigcup_{C\in  \mca{V}_{a_i}\cup \mca{V}_{b_i}}V(C).$$
Then $U_1,\ldots, U_{c_{\ref{lem:mainedgecluster}}}\subseteq V(G)$ and $E(G[U_1])\cup \cdots \cup E(G[U_{c_{\ref{lem:mainedgecluster}}}])=E(G)$. Furthermore, for each $j\in \{1,\ldots, c_{\ref{lem:mainedgecluster}}\}$, since every component of $\Gamma[\mca{V}_{a_j}\cup \mca{V}_{b_j}]$ is a star, and $\diam_G(C)< 4\sigma$ for every $C\in \mca{V}_{a_j}\cup \mca{V}_{b_j}\subseteq \mca{U}$, it follows that every component $D$ of the graph $(V(G),E(G[U_j]))$ satisfies $\diam_G(V(D))\leq 3(4\sigma-1)+2<12\sigma$. 

Let $E_1=E(G[U_1])$, and for every $j\in \{2,\ldots, c_{\ref{lem:mainedgecluster}}\}$, let $E_j=E(G[U_j])\setminus (E_1\cup \cdots\cup E_{j-1})$.
Then $(E_1,\ldots, E_{c_{\ref{lem:mainedgecluster}}})$ is a $c_{\ref{lem:mainedgecluster}}$-edge-coloring of $G$. Moreover, for each $j\in \{1, \ldots, c_{\ref{lem:mainedgecluster}}\}$, since $E_{j}\subseteq E(G[U_j])$, it follows that for every component $D$ of the graph $(V(G),E_j)$, there is a component $D'$ of the graph $(V(G),E(G[U_j]))$ such that $V(D)\subseteq V(D')$, and so $\diam_G(V(D))\leq \diam_G(V(D'))<12\sigma$. Hence, $(E_1,\ldots, E_{c_{\ref{lem:mainedgecluster}}})$ has $G$-diameter less than $12\sigma$. This completes the proof of \Cref{lem:mainedgecluster}.
\end{proof}

\subsection{Induced minor models of bipartite graphs}\label{subsec:indmtorso}

In this subsection, we prove two technical lemmas about induced minor models of bipartite graphs with certain constraints. The first lemma has a fairly simple proof:

\begin{lemma}\label{lem:1-sub}
    Let $H$ be a bipartite graph with a bipartition $(A,B)$ such that every vertex in $A$ has degree two in $H$. Let $G$ be a graph with an induced minor isomorphic to $H$. Then there is an induced $H$-model $(X_v:v\in V(H))$ in $G$ with $|V(X_a)|=1$ for all $a\in A$.
\end{lemma}
\begin{proof}
  Let $(X_v:v\in V(H))$ be an induced $H$-model in $G$ with $|\bigcup_{a\in A}V(X_a)|$ as small as possible. We claim that $|V(X_a)|=1$ for all $a\in A$. Suppose not. Then there exists $a_0\in A$ such that $|V(X_{a_0})|\geq 2$. Let $N_H(a_0)=\{b_0,b_1\}\subseteq B$. For every $v\in V(H)\setminus \{a_0,b_0\}$, let $Y_v=X_v$ (in particular, we have $Y_{b_1}=X_{b_1}$). We show that:

   \sta{\label{st:minimalmodel} For every connected induced subgraph $Y$ of $X_{a_0}$ with $|V(Y)|<|V(X_{a_0})|$, either $X_{b_0}$ and $Y$ are anticomplete in $G$ or $X_{b_1}$ and $Y$ are anticomplete in $G$.}

Suppose $X_{b_0}$ and $Y$ are not anticomplete in $G$, and $X_{b_1}$ and $Y$ are not anticomplete in $G$. Let $Y_{a_0}=Y$ and let $Y_{b_0}=X_{b_0}$; in particular, we have $|V(Y_{a_0})|<|V(X_{a_0})|$. Then it is readily seen that $(Y_v:v\in V(H))$ is an induced $H$-model in $G$ with $|\bigcup_{a\in A}V(Y_a)|< |\bigcup_{a\in A}V(X_a)|$, a contradiction. This proves \eqref{st:minimalmodel}.
\medskip

Since $a_0b_1\in E(H)$, it follows that there is a vertex $x_1\in V(X_{a_0})$ that has a neighbor in $G$ in $V(X_{b_1})$. By \eqref{st:minimalmodel} and since $|V(X_{a_0})|\geq 2$, it follows that $x_1$ is anticomplete to $V(X_{b_0})$ in $G$. Thus, since $a_0b_0\in E(H)$, there is a vertex $x_0\in V(X_{a_0}\setminus x_1)$ that has a neighbor in $G$ in $V(X_{b_0})$. Let $C$ be the component of $X_{a_0}\setminus x_1$ containing $x_0$. Then, by \eqref{st:minimalmodel}, $C$ and $V(X_{b_1})$ are anticomplete in $G$. Now, let $Y_{a_0}=G[\{x_1\}]$ and let $Y_{b_0}=G[V(C)\cup V(X_{b_0})]$; in particular, $|V(Y_{a_0})|\leq |V(X_{a_0})|-|V(C)|<|V(X_{a_0})|$. Then, $(Y_v:v\in V(H))$ is an induced $H$-model in $G$ with $|\bigcup_{a\in A}V(Y_a)|< |\bigcup_{a\in A}V(X_a)|$, a contradiction. This completes the proof of \Cref{lem:1-sub}.
\end{proof}

For our second lemma, we need several definitions. Given a tree $T$ and distinct vertices $x,y\in V(T)$, we denote by $T_{x\setminus y}$ the component of $T\setminus y$ that contains $x$. Let $G$ be a graph and let $\mf{t}=(T,\beta)$ be a tree decomposition of $G$. 

\begin{itemize}
    \item For every $S\subseteq V(T)$, we write $\beta(S)=\bigcup_{x\in S}\beta(x)$, and for every edge $xy\in E(T)$, we write $\beta(x,y)=\beta(x)\cap \beta(y)=\beta(V(T_{x\setminus y}))\cap \beta(V(T_{y\setminus x}))$ (the latter equality follows easily from the definition of a tree decomposition). 
    
    \item For $x\in V(T)$, the \textit{torso of $\mf{t}$ at $x$}, denoted $\mf{t}_{x}$, is the graph with $V(\mf{t}_{x})=\beta(x)$ such that for all distinct $u,v\in \beta(x)$, we have $uv\in E(\mf{t}_{x})$ if and only if either $uv\in E(G)$ or $u,v\in \beta(x,y)$ for some $y\in N_T(x)$. So $\mf{t}_{x}$ is obtained from $G[\beta(x)]$ by adding an edge between every two nonadjacent vertices in $G[\beta(x,y)]$ for all $y\in N_T(x)$.
    
    \item For $x\in V(T)$, the \textit{extended torso of $\mf{t}$ at $x$}, denoted $\mf{t}^{\ex}_{x}$, is the graph with $V(\mf{t}^{\ex}_{x})=V(G)$ such that for all distinct $u,v\in V(G)$, we have $uv\in E(\mf{t}^{\ex}_{x})$ if and only if either $uv\in E(G)$ or $u,v\in \beta(x,y)$ for some $y\in N_T(x)$. So $\mf{t}^{\ex}_{x}$ is obtained from $G$ by adding an edge between every pair of nonadjacent vertices in $G[\beta(x,y)]$ for all $y\in N_T(x)$. It follows that $\mf{t}_x=\mf{t}^{\ex}_x[\beta(x)]$.
\item We say that $\mf{t}$ is \textit{tight} if for each $x\in V(T)$ and every $y\in N_T(x)$, there is a component $C$ of $G[\beta(V(T_{y\setminus x}))\setminus \beta(V(T_{x\setminus y}))]$ such that every vertex in $\beta(x,y)$ is adjacent in $G$ to at least one vertex in $V(C)$. 
\end{itemize}

\begin{lemma}\label{lem:brambletoindm}
     Let $H$ be a bipartite graph with a bipartition $(A,B)$. Let $\mf{t}=(T,\beta)$ be a tight tree decomposition of a graph $G$ and let $x\in V(T)$ such that there is an induced $H$-model $\mf{m}=(X_v:v\in V(H))$ in $\mf{t}_x$ with $|V(X_a)|=1$ for all $a\in A$. Then $G$ has an induced minor isomorphic to $H$.
\end{lemma}
\begin{proof}
 For each $y\in N_T(x)$, choose a component $C_y$ of $G[\beta(V(T_{y\setminus x}))\setminus \beta(V(T_{x\setminus y}))]$ such that every vertex in $\beta(x,y)$ has a neighbor in $G$ in $V(C_y)$ (this is possible as $\mf{t}$ is tight). Then, for all distinct $u_1,u_2\in \beta(x,y)$ with $u_1u_2\notin E(G)$, there is a path in $G$ from $u_1$ to $u_2$ whose interior is contained in $V(C_y)$. Also, since $(T,\beta)$ is a tree decomposition of $G$, it is readily seen that:

\sta{\label{st:tightproperties}$(C_y: y\in N_T(x))$ are pairwise anticomplete in $G$, and for every $y\in N_T(x)$, $\beta(x)\setminus \beta(y)$ and $V(C_y)$ are anticomplete in $G$.}

Next, for every $v\in V(H)$, let 
$$S_v=\{y\in N_T(x): V(X_v)\cap \beta(y)\neq \varnothing\}, \quad U_v=\bigcup_{y\in S_v} V(C_y), \quad Y_v=G[V(X_v)\cup U_v].$$
We show that:

\sta{\label{st:connectedbramble} $Y_v$ is connected for all $v\in V(H)$.}

Suppose that for some $v\in V(H)$, $Y_v$ has two distinct components $D_1,D_2$. Since $X_v$ is a non-null and connected induced subgraph of $\mf{t}_{x}$, and from the definitions of $S_v, U_v$ and $Y_v$, it follows that $\mf{t}^{\ex}_{x}[V(Y_v)]$ is connected, and so there are vertices $u_1\in V(D_1)$ and $u_2\in V(D_2)$ such that $u_1u_2\in E(\mf{t}^{\ex}_{x})$. Since $u_1u_2\notin E(G)$, it follows that for some $y\in N_T(x)$, we have $u_1,u_2\in \beta(x,y)$; in particular, $u_1,u_2\in V(X_v)\cap \beta(y)$, which in turn implies that $y\in S_v$, and so there is a path $P$ in $G$ from $u_1$ to $u_2$ with $V(P^*)\subseteq V(C_y)\subseteq U_v$. But then $V(P)\subseteq V(X_v)\cup U_v=V(Y_v)$, and thus $P$ is a path in $Y_v$ from $u_1$ to $u_2$, contrary to the assumption that $u_1$ and $u_2$ belong to distinct components of $Y_v$. This proves \eqref{st:connectedbramble}.

\sta{\label{st:antibramble} $Y_{v_1}$ and $Y_{v_2}$ are anticomplete in $G$ for all distinct $v_1,v_2\in V(H)$ with $v_1v_2\notin E(H)$.}

Since $\mf{m}$ is an induced $H$-model in $\mf{t}_x$, it follows that $V(X_{v_1})$ and $V(X_{v_2})$ are anticomplete in $\mf{t}_x$, and so in $G$ (because $E(G[\beta(x)])\subseteq E(\mf{t}_x)$). In particular, since $\beta(x,y)$ is a clique in $\mf{t}_{x}$ for every $y\in N_T(x)$, it follows that $S_{v_1}\cap S_{v_2}=\varnothing$. So by \eqref{st:tightproperties} and the definitions of $S_{v_1}, S_{v_2}$ and $U_{v_1}, U_{v_2}$, we have that $U_{v_1},U_{v_2}$ are anticomplete in $G$, and so are $V(X_{v_1}),U_{v_2}$ and $V(X_{v_2}),U_{v_1}$. This proves \eqref{st:antibramble}.

\sta{\label{st:AnotinB} For all $a\in A$ and $b\in B$ with $ab\in E(H)$, we have $V(X_{a})\cap V(Y_b)=\varnothing$, but $V(X_a)$ and $V(Y_b)$ are not anticomplete in $G$.}

Let $V(X_a)=\{u_a\}$. Since $\mf{m}$ is an induced $H$-model in $\mf{t}_x$, it follows that $u_a\notin V(X_b)$, and there exists $w\in V(X_b)$ such that $u_aw\in E(\mf{t}_x)$. Since $u_a\in V(\mf{t}_x)=\beta(x)$ and $V(Y_b)\setminus V(X_b)=U_b\subseteq V(G)\setminus \beta(x)$, it follows that $u_{a}\notin V(Y_b)$, and so $V(X_{a})\cap V(Y_b)=\varnothing$. It remains to show that $V(X_a)$ and $V(Y_b)$ are not anticomplete in $G$. If $u_{a}w\in E(G)$, then we are done. Therefore, we may assume that $u_{a}w\notin E(G)$, and so for some $y\in N_T(x)$, we have $u_a,w\in \beta(x,y)$. But now $y\in S_b$, and so $u_a\in V(X_{a})\cap\beta(x,y)$ has a neighbor in $G$ in $V(C_y)\subseteq U_b\subseteq V(Y_b)$, as desired. This proves \eqref{st:AnotinB}.
\medskip

For every $v\in V(H)$, define $Z_v$ as follows: If $v\in A$, then let $Z_v=X_v$, and if $v\in B$, then let $Z_v=Y_v$. It follows that $Z_v$ is always an induced subgraph of $Y_v$. Now, from \eqref{st:connectedbramble}, \eqref{st:antibramble}, and \eqref{st:AnotinB}, it follows that $(Z_v:v\in V(H))$ is an induced $H$-model in $G$. This completes the proof of \Cref{lem:brambletoindm}.
\end{proof}

\subsection{Excluding subdivided complete graphs as subgraphs}\label{subsec:subcompletesubg} Here, we take the penultimate step toward the proof of \Cref{thm:mainpolyblock} by showing that:

\begin{lemma}\label{lem:subcompletesubg}
 For every planar graph $H$ and all $\sigma\in \poi$, there is a constant $d_{\ref{lem:subcompletesubg}}=d_{\ref{lem:subcompletesubg}}(H,\sigma)\in \poi$ such that for all $\mu\in \poi$, every graph $G$ with no subgraph isomorphic to any subdivision of $K_{\mu}$ and no induced minor isomorphic to $H$ or $K_{\sigma,\sigma}$ satisfies 
    $\tw(G)< \mu^{d_{\ref{lem:subcompletesubg}}}$.
\end{lemma}

The proof follows a strategy used in our previous paper \cite{tw7} with Abrishami, Alecu, Chudnovsky and Spirkl (building in turn on ideas from \cite{lozin}). 

In particular, we will use the following, which, as pointed out in \cite[6.2, 6.3, and 6.4]{tw7}, is an immediate corollary of Theorem~4 from \cite{tighttw} combined with Lemma~6 from \cite{Weissauerblock}:

\begin{theorem}[Erde and Wei{\ss}auer {\cite[Theorem~4]{tighttw}}; Wei{\ss}auer {\cite[Lemma~6]{Weissauerblock}}]\label{thm:tighttd}
    For all $\mu\in \poi$, every graph $G$ with no subgraph isomorphic to any subdivision of $K_{\mu}$ admits a tight tree decomposition $\mf{t}=(T,\beta)$ such that for every $x\in V(T)$, either $\mf{t}_{x}$ has fewer than $\mu^2$ vertices of degree at least $2\mu^4$, or $\mf{t}_{x}$ has no minor isomorphic to $K_{2\mu^2}$.
\end{theorem}

We also need a few other results from the literature:

\begin{lemma}[See e.g. {\cite[Lemma~5]{cliquetw}}]\label{lem:bodtorso}
Let $G$ be a graph and let $\mf{t}=(T,\beta)$ be a tree decomposition of $G$. Then
$\tw(G)=\max_{x\in V(T)}\tw\left(\mf{t}_{x}\right)$.
\end{lemma}

\begin{theorem}[Campbell, Davies, Distel, Frederickson, Gollin, Hendrey, Hickingbotham, Wiederrecht, Wood, Yepremyan {\cite[Theorem~17]{twhad}}]\label{thm:gridtoinducedgrid}
 Let $\nu,\tau\in \poi$ and let $G$ be a graph with a minor isomorphic to the $(2\nu+1)\tau$-grid. Then $G$ has either an induced minor isomorphic to the $\nu$-grid or a minor isomorphic to $K_{\tau}$.
\end{theorem}

\begin{theorem}[Chuzhoy and Tan \cite{chtan}]\label{thm:bettergrid}
 There exists $c_{\ref{thm:bettergrid}}\in \poi$ such that for all $\nu\in \poi$, every graph $G$ with no minor isomorphic to the $\nu$-grid satisfies $\tw(G)\leq c_{\ref{thm:bettergrid}}\nu^9\log^{c_{\ref{thm:bettergrid}}}\nu$.
\end{theorem}

We may now prove \Cref{lem:subcompletesubg}:

\begin{proof}[Proof of \Cref{lem:subcompletesubg}]
    Define $H_1$ to be the graph with $V(H_1)=E(H)\cup V(H)$ and $E(H_1)=\{ev:e\in E(H), v\in V(H), v\in e\}$. Then $H_1$ is isomorphic to the $1$-subdivision of $H$; thus, $H_1$ is planar. Also, $H_1$ is bipartite with bipartition $(E(H),V(H))$ where every ``vertex'' $e\in E(H)$ has degree two in $H_1$. Let $h=|V(H_1)|=|E(H)|+|V(H)|$, let 
    $c=c_{\ref{cor:compbip}}(1,h,\sigma,\sigma,\sigma)$ and let $d=d_{\ref{lem:mainpolydegree}}(h,c)$.
    Recall the constant $c_{\ref{thm:bettergrid}}\in \poi$. We show that
$$d_{\ref{lem:subcompletesubg}}=d_{\ref{lem:subcompletesubg}}(H,\sigma)=\max\left\{5d+1, c_{\ref{thm:bettergrid}}+(56h+3)(9+c_{\ref{thm:bettergrid}})\right\}$$
satisfies the lemma. 

Let $G$ be a graph with no subgraph isomorphic to any subdivision of $K_{\mu}$ and no induced minor isomorphic to $H$ or $K_{\sigma,\sigma}$. If $\mu=1$, then $\tw(G)=0<\mu^{d_{\ref{lem:subcompletesubg}}}$. So we may assume that $\mu\geq 2$. By \Cref{thm:tighttd}, there is a tight tree decomposition $\mf{t}=(T,\beta)$ of $G$ such that for every $x\in V(T)$, either $\mf{t}_{x}$ has fewer than $\mu^2$ vertices of degree at least $2\mu^4$, or $\mf{t}_{x}$ has no minor isomorphic to $K_{2\mu^2}$. By \Cref{lem:bodtorso}, there is a vertex $x_0\in V(T)$ for which $\tw(\mf{t}_{x_0})=\tw(G)$. We further show that:

\sta{\label{st:noH1indm} $\mf{t}_{x_0}$ has no induced minor isomorphic to $H_1$.}

Suppose not. Then, by \Cref{lem:1-sub}, there is an induced $H_1$-model $(X_{w}:w\in V(H_1))$ in $\mf{t}_{x_0}$ such that $|V(X_e)|=1$ for all $e\in E(H)$. Since $\mf{t}$ is tight, it follows from \Cref{lem:brambletoindm} that $G$ has an induced minor isomorphic to $H_1$, and so $G$ has an induced minor isomorphic to $H$ (as $H_1$ is isomorphic to the $1$-subdivision of $H$), a contradiction. This proves \eqref{st:noH1indm}.

\sta{\label{st:nocompbipindm} $\mf{t}_{x_0}$ has no induced minor isomorphic to $K_{c,c}$.}

Suppose not. Then, by the choice of $c$, we can apply \Cref{cor:compbip} to $\mf{t}_{x_0}$. By \eqref{st:noH1indm}, $\mf{t}_{x_0}$ has no induced minor isomorphic to the $h$-vertex planar graph $H_1$, and so \ref{cor:compbip}\ref{cor:compbip_b} does not hold. Consequently, either $\mf{t}_{x_0}$ has an induced subgraph isomorphic to $K_{\sigma,\sigma}$, or there is a $1$-ample $(\sigma,\sigma)$-constellation in $\mf{t}_{x_0}$. In either case, it follows that there is an induced $K_{\sigma,\sigma}$-model $\mf{m}=(A_1,\ldots, A_{\sigma}; B_1,\ldots, B_{\sigma})$ in $\mf{t}_{x_0}$ where $|V(A_1)|=\cdots=|V(A_{\sigma})|=1$. But now, since $\mf{t}$ is tight, by \Cref{lem:brambletoindm}, $G$ has an induced minor isomorphic to $K_{\sigma,\sigma}$, a contradiction. This proves \eqref{st:nocompbipindm}.
\medskip

Now, recall that $\mu\geq 2$. If $\mf{t}_{x_0}$ has fewer than $\mu^2$ vertices of degree at least $2\mu^4$, then by \eqref{st:noH1indm}, \eqref{st:nocompbipindm}, and \Cref{lem:mainpolydegree}, we have
$\tw(\mf{t}_{x_0})\leq (2\mu^4)^{d}+\mu^2\leq \mu^{5d}+\mu^2< 2\mu^{5d}\leq \mu^{5d+1}\leq \mu^{d_{\ref{lem:subcompletesubg}}}$. Thus, we may assume that $\mf{t}_{x_0}$ has no minor isomorphic to $K_{2\mu^2}$. By \eqref{st:noH1indm} and Theorems~\ref{thm:gridtograph} and \ref{thm:gridtoinducedgrid}, $\mf{t}_{x_0}$ has no minor isomorphic to the $2\mu^2(56h+1)$-grid. Let $\nu=2\mu^2(56h+1)$. Then by \Cref{thm:bettergrid}, we have $\tw(G)\leq c_{\ref{thm:bettergrid}}\nu^9\log^{c_{\ref{thm:bettergrid}}}\nu$. Moreover, $c_{\ref{thm:bettergrid}}<2^{c_{\ref{thm:bettergrid}}}\leq \mu^{c_{\ref{thm:bettergrid}}}$ and $\nu=2\mu^2(56h+1)\leq \mu^3\cdot 2^{56h}\leq \mu^{56h+3}$. Hence, 
$$\tw(\mf{t}_{x_0})< \mu^{c_{\ref{thm:bettergrid}}}\cdot \left(\mu^{56h+3}\right)^{9+c_{\ref{thm:bettergrid}}}\leq \mu^{d_{\ref{lem:subcompletesubg}}}.$$
This completes the proof of \Cref{lem:subcompletesubg}.
\end{proof}

\subsection{Proof of \Cref{thm:mainpolyblock}}\label{subsec:finalblock} We need a better version of \Cref{thm:ramsey} for graphs that exclude a fixed complete bipartite induced subgraph, which we will also use in \Cref{sec:mainpolysep}. (For the familiar reader, this simply restates that the Erd\H{o}s-Hajnal conjecture \cite{EH} is true for complete bipartite graphs, which, for instance, follows from the main results of \cite{APS} and \cite{EH2}. For completeness, we include a short proof with an explicit bound.)

\begin{lemma}\label{lem:EHforKss}
    For all $\alpha,s,t\in \poi$, every $K_{s,s}$-free and $K_{t}$-free graph $G$ on at least $\alpha^s t^{s-1}$ vertices has a stable set of cardinality $\alpha$.
\end{lemma}
\begin{proof}
This is immediate if $s=1$. Also, if $\alpha\leq s$, then the result follows from \Cref{thm:ramsey}. Assume that $\alpha\geq s\geq 2$. Let $S$ be a stable set in $G$ of maximum cardinality. For every $A\in \binom{S}{s-1}$, let $X_A$ be the set of all vertices in $V(G)\setminus S$ that are anticomplete to $S\setminus A$ in $G$. Then $|X_A|\leq t^{s-1}-1$; because otherwise, by \Cref{thm:ramsey} applied to $G[X_A]$, there is a stable set $S'\subseteq X_A$ in $G$ with $|S'|=s$, and since $S'\subseteq X_A$ and $S\setminus A$ are anticomplete in $G$, it follows that $(S\setminus A)\cup S'$ is a stable set in $G$ of cardinality $|S|+1$, a contradiction. For every $B\in \binom{S}{s}$, let $Y_B$ be the set of all vertices in $V(G)\setminus S$ that are adjacent in $G$ to all the vertices in $B$. Then $|Y_B|\leq t^{s-1}-1$; because otherwise, by \Cref{thm:ramsey} applied to $G[Y_B]$, there is a stable set $S'\subseteq Y_B$ in $G$ with $|S'|=s$, and since every vertex in $S'\subseteq Y_B$ is adjacent in $G$ to every vertex in $B$, it follows that $G[S'\cup B]$ is isomorphic to $K_{s,s}$, a contradiction. Note also that
$V(G)\setminus S= (\bigcup_{A\in \binom{S}{s-1}}X_A)\cup (\bigcup_{B\in \binom{S}{s}}Y_B)$. Therefore,
$$|V(G)|\leq |S|+\left(\binom{|S|}{s-1}+\binom{|S|}{s}\right)(t^{s-1}-1)=|S|+\binom{|S|+1}{s}(t^{s-1}-1)<(|S|+1)^st^{s-1}.$$
Since $|V(G)|\geq \alpha^{s}t^{s-1}$, we obtain $|S|\geq \alpha$. This completes the proof of \Cref{lem:EHforKss}.
\end{proof}

Finally, we can prove \Cref{thm:mainpolyblock}, which we restate:

\mainpolyblock*
\begin{proof}
Let $d_{\ref{thm:mainpolyblock}}=d_{\ref{thm:mainpolyblock}}(H,\sigma)=2\sigma d_{\ref{lem:subcompletesubg}}(H,\sigma)$. Let $G$ be a $(\kappa,\lambda)$-separable $K_t$-free graph with no induced minor isomorphic to $H$ or $K_{\sigma,\sigma}$. Let $\tau=\kappa^{\sigma}t^{\sigma-1}$. Assume that $G$ has a subgraph isomorphic to some subdivision of $K_{\lambda\tau^2}$. Then, since $2\lambda\tau^2\geq \lambda\tau^2+\lambda\tau=2\lambda\tau+\lambda\tau^2-\lambda\tau\geq 2\tau+\lambda\tau(\tau-1)$, it follows that $\lambda\tau^2\geq\tau+\lambda\binom{\tau}{2}$, and so there is a strong $(\tau,\lambda)$-block $(B,\mca{P})$ in $G$. Also, since $G$ is $K_{\sigma,\sigma}$-free and $K_{t}$-free, by \Cref{lem:EHforKss} and the choice of $\tau=\kappa^{\sigma}t^{\sigma-1}$, there is a stable set $B'\subseteq B$ in $G$ with $|B'|=\kappa$. But now $(B',\mca{P}|\binom{B'}{2})$ is a stable strong $(\kappa,\lambda)$-block in $G$, contrary to the $(\kappa,\lambda)$-separability of $G$. Thus, $G$ has no subgraph isomorphic to any subdivision of $K_{\lambda\tau^2}$, and so by \Cref{lem:subcompletesubg}, $\tw(G)\leq (\lambda\tau^2)^{d_{\ref{lem:subcompletesubg}}(H,\sigma)}\leq (\kappa \lambda t)^{2\sigma d_{\ref{lem:subcompletesubg}}(H,\sigma)}=(\kappa \lambda t)^{d_{\ref{thm:mainpolyblock}}}$.
This completes the proof of \Cref{thm:mainpolyblock}.
\end{proof}

\section{Polynomial Separability}\label{sec:mainpolysep}

In this final section, we prove \Cref{thm:mainpolysep}. To that end, we adopt the ``seedling method'' that we developed in \cite{tw18} with Chudnovsky and Spirkl. 

Let $G$ be a graph. For $X,Y\subseteq V(G)$, an \textit{$(X,Y)$-linkage in $G$} is a set $\mca{L}$ of pairwise disjoint paths in $G$ such that for every $L\in \mca{L}$, either $L$ has length zero and $V(L)\subseteq X\cap Y$, or $L$ has nonzero length, one end of $L$ belongs to $X\setminus Y$, the other end belongs to $Y\setminus X$, and $V(L^*)\cap (X\cup Y)=\varnothing$. In particular, every path $L\in \mca{L}$ has an end in $X$ and an end in $Y$, respectively called the \textit{$X$-end} and the \textit{$Y$-end of $L$} (which are the same if and only if $|V(L)|=1$). We say that $\mca{L}$ is \textit{$\mu$-rigid}, for $\mu\in \poi$, if there is no $(X,Y)$-linkage $\mca{M}$ in $G$ with $|\mca{M}|=\mu$ and $V(\mca{M})\subseteq V(\mca{L})$ such that the paths in $\mca{M}$ are pairwise anticomplete in $G$. For $\lambda,\mu\in \poi$, a \textit{$(\lambda,\mu)$-seedling in $G$} is a triple $(A,\mca{L},Y)$ such that $A$ is a path in $G$, $Y\subseteq V(G)\setminus V(A)$, and $\mca{L}$ is a $\mu$-rigid $(N_G(V(A)),Y)$-linkage in $G\setminus V(A)$ with $|\mca{L}|=\lambda$; it follows that $V(A)\cap V(\mca{L})=\varnothing$, and for every $L\in \mca{L}$, the $N_G(V(A))$-end of $L$ is the only vertex of $L$ with a neighbor in $G$ in $V(A)$. Two $(\lambda,\mu)$-seedlings $(A,\mca{L},Y)$ and $(A',\mca{L}',Y')$ in $G$ are \textit{disjoint} if $(A\cup V(\mca{L}))\cap (A'\cup V(\mca{L}'))=\varnothing$.

The following technical lemma is the main ingredient in the proof of \Cref{thm:mainpolysep}:

\begin{lemma}\label{lem:getbranches}
For all $\theta,\sigma\in \poi$, there is a constant $c_{\ref{lem:getbranches}}=c_{\ref{lem:getbranches}}(\theta,\sigma)\in \poi$ with the following property. Let $\lambda,\mu,t\in \poi$ with $t\geq 2$, let $G$ be a $K_{t}$-free graph with no induced minor isomorphic to $K_{\sigma,\sigma}$, and let $(A,\mca{L},Y)$ be a $\left((\lambda t)^{\sigma(\sigma+3)\left(\theta+2\mu\sigma\right)(\mu-1)},\mu\right)$-seedling in $G$. Then there are $\theta$ pairwise disjoint $(\lambda, c_{\ref{lem:getbranches}})$-seedlings $(A_1,\mca{L}_1,Y_1),\ldots, (A_{\theta},\mca{L}_{\theta},Y_{\theta})$ in $G[V(\mca{L})]$ such that $A_1,\ldots, A_{\theta}$ are pairwise anticomplete in $G$, and for every $i\in \{1,\ldots, \theta\}$, $A$ and $A_i$ are not anticomplete in $G$ but $V(A)$ and $V(\mca{L}_i)$ are anticomplete in $G$.
\end{lemma}

The proof of \Cref{lem:getbranches} uses two Ramsey-type results from \cite{tw18}. The first is about digraphs (and is easy to prove). For us, a \textit{digraph} $D$ has a finite vertex set $V(D)$ and an edge set $E(D)$ whose elements are ordered pairs of distinct vertices of $D$ (so we disallow ``loops'' and allow at most one edge in either direction between each pair of vertices). For $(u,v)\in E(D)$, we say that $v$ is an \textit{out-neighbor} of $u$. The \textit{out-degree} of $u\in V(D)$ is the number of out-neighbors of $u$. For $r\in \poi$, we write $D^{\leq r}$ and $D^{\geq r}$, respectively, for the set of all vertices in $D$ of out-degree at most $r$ and those of out-degree at least $r$. A \textit{stable set in $D$} is a subset $S$ of $V(D)$ such that $(u,v)\notin E(D)$ and $(v,u)\notin E(D)$ for all $u,v\in S$.

\begin{lemma}[Chudnovsky, Hajebi, Spirkl {\cite[Lemma~5.1]{tw18}}]\label{lem:digraph}
    The following hold for all $q,r,s\in \poi$ and every digraph $D$.
    \begin{enumerate}[{\rm (a)}]
    \item \label{lem:digraph_a} 
    If $|D^{\leq r}|\geq 2rs$, then there is a stable set in $D$ of cardinality $s$.
    \item \label{lem:digraph_b} If $|D^{\geq qr}|\geq 2qrs$, then there is an $s$-subset $S$ of $V(D)$ such that for every $q$-subset $\{v_1,\ldots, v_q\}$ of $S$, there are pairwise disjoint $r$-subsets $R_1,\ldots, R_q$ of $V(D)\setminus S$ where for each $i\in \{1,\ldots, q\}$, the vertices in $R_i$ are out-neighbors of $v_i$ in $D$.
    \end{enumerate}
\end{lemma}

The second is the following (the proof combines \Cref{lem:digraph} with \Cref{thm:productramsey} below).

\begin{lemma}[Chudnovsky, Hajebi, Spirkl {\cite[Lemma~5.4]{tw18}}]\label{lem:bigramsey}
    For all $r,s,\sigma\in \poi$, there is a constant $c_{\ref{lem:bigramsey}}=c_{\ref{lem:bigramsey}}(r,s,\sigma)\in \poi$ with the following property. Let $G$ be a graph with no induced minor isomorphic to $K_{\sigma,\sigma}$ and let 
    $\{A_{i,0},A_{i,1},\ldots, A_{i,c_{\ref{lem:bigramsey}}}:i\in \{1,\ldots, 2r\sigma\}\}$
    be a set of $2r\sigma(c_{\ref{lem:bigramsey}}+1)$ pairwise disjoint connected induced subgraphs of $G$ such that:
    \begin{itemize}
        \item $A_{1,0},\ldots, A_{2r\sigma,0}$ are pairwise anticomplete in $G$; and 
        \item $A_{i,1},\ldots, A_{i,c_{\ref{lem:bigramsey}}}$ are pairwise anticomplete in $G$ for each $i\in \{1,\ldots, 2r\sigma\}$.
    \end{itemize}
    Then there is an $r$-subset $I$ of $\{1,\ldots, 2r\sigma\}$ and, for each $i\in I$, an $s$-subset $\{A'_{i,1},\ldots,A'_{i,s}\}$ of $\{A_{i,1},\ldots, A_{i,c_{\ref{lem:bigramsey}}}\}$ (so $c_{\ref{lem:bigramsey}}\geq s$) such that the sets $(A_{i,0}\cup A'_{i,1}\cup \cdots \cup A'_{i,s}:i\in I)$ are pairwise anticomplete in $G$.
\end{lemma}

\begin{proof}[Proof of \Cref{lem:getbranches}]
We will show that
$$c_{\ref{lem:getbranches}}=c_{\ref{lem:getbranches}}(\theta,\sigma)=c_{\ref{lem:bigramsey}}(\theta,1,\sigma)$$
satisfies the lemma.
 
Let $\Theta=\theta+2\mu\sigma-1$. Then $|\mca{L}|=(\lambda t)^{\sigma(\sigma+3)\left(\Theta+1\right)(\mu-1)}$. Since $\mca{L}$ is $\mu$-rigid, no $\mu$ paths in $\mca{L}$ are pairwise anticomplete in $G$, and so by \Cref{thm:ramsey}, there is a $(\lambda t)^{\sigma(\sigma+3)\left(\Theta+1\right)}$-subset $\mca{S}_0$ of $\mca{L}$ such that no two paths $L,L'\in \mca{S}_0$ are anticomplete in $G$. In particular:
$$|\mca{S}_0|=(\lambda t)^{\sigma((\sigma+3)+(\Theta-1)(\sigma+3))}\cdot (\lambda t)^{\sigma(\sigma+3)}.$$
  
Since $t\geq 2$, $\sigma+3\geq 4$, $16>9$, and $2^{\Theta-1}\geq \Theta$, it follows that:
$$(\lambda t)^{\sigma((\sigma+3)+(\Theta-1)(\sigma+3))}\geq 2^{\sigma(4+(\Theta-1)(\sigma+3))}=\left(16\left(2^{\Theta-1}\right)^{\sigma+3}\right)^{\sigma}>\left(9\Theta^{\sigma+3}\right)^{\sigma}.$$

Also, note that $(\lambda t)^{\sigma(\sigma+3)}>\lambda^{\sigma(\sigma+3)} t^{\sigma^2-1}=\left(\lambda^{\sigma+3} t^{\sigma-1}\right)^{\sigma} t^{\sigma-1}$. We deduce that:
$$|\mca{S}_0|\geq \left(9\Theta^{\sigma+3}\right)^{\sigma}\cdot \left(\lambda^{\sigma+3} t^{\sigma-1}\right)^{\sigma} t^{\sigma-1}=\left(9(\Theta\lambda)^{\sigma+3}t^{\sigma-1}\right)^{\sigma} t^{\sigma-1}.$$

For each $L\in \mca{S}_0\subseteq \mca{L}$, let $x_L$ and $y_L$, in order, be the $N_G(V(A))$-end and the $Y$-end of $L$. Since $|\mca{S}_0|\geq \left(9(\Theta\lambda)^{\sigma+3}t^{\sigma-1}\right)^{\sigma} t^{\sigma-1}$ and $G$ is $K_{\sigma,\sigma}$-free and $K_t$-free, by \Cref{lem:EHforKss}, there is a $9(\Theta\lambda)^{\sigma+3}t^{\sigma-1}$-subset $\mca{S}_1$ of $\mca{S}_0$ for which $\{x_L:L\in \mca{S}_1\}$ is a stable set in $G$.

Let $D_1$ be the digraph with vertex set $\mca{S}_1$ such that for all distinct $L,L'\in \mca{S}_1$, we have $(L,L')\in E(D_1)$ if and only if $x_L$ has a neighbor in $G$ in $V(L')$.

\sta{\label{st:manyoutend} If $|D_1^{\geq \Theta\lambda}|\geq 2\Theta^2\lambda$, then there exists a $\Theta$-subset $\mca{P}$ of $\mca{S}_1$, and $\Theta$ pairwise disjoint $\lambda$-subsets $(\mca{R}_L:L\in \mca{P})$ of $\mca{S}_1\setminus \mca{P}$, such that for each $L\in \mca{P}$ and every $R\in \mca{R}_L$, $x_L$ is not anticomplete to $V(R)\setminus \{x_R\}$ in $G$.}

By \ref{lem:digraph}\ref{lem:digraph_b} applied to $D_1$ with $r=\lambda$ and $q=s=\Theta$, there is a $\Theta$-subset $\mca{P}$ of $\mca{S}_1$ as well as $\Theta$ pairwise disjoint $\lambda$-subsets $(\mca{R}_L:L\in \mca{P})$ of $\mca{S}_1\setminus \mca{P}$, such that for every $L\in \mca{P}$, $x_{L}$ has a neighbor in $G$ in each path in $\mca{R}_{L}$. Also, $\{x_L:L\in \mca{P}\}$ is a stable set in $G$ because $\{x_{L}:L\in \mca{S}_1\}$ is. Thus, for each $L\in \mca{P}$ and every $R\in \mca{R}_L$, $x_L$ has a neighbor in $G$ in $V(R)\setminus \{x_R\}$. This proves \eqref{st:manyoutend}.

\sta{\label{st:fewoutend} If $|D_1^{\geq \Theta\lambda}|\leq 2\Theta^2\lambda$, then there exists a $4(\Theta\lambda)^{\sigma+2}t^{\sigma-1}$-subset $\mca{S}_2$ of $\mca{S}_1$ such that for all distinct $L,L'\in \mca{S}_2$, $x_{L}$ is anticomplete to $V(L')$ in $G$.}

Since $\Theta>1$, it follows that $|D_1^{\leq \Theta\lambda}|\geq |\mca{S}_1|-2\Theta^2\lambda>|\mca{S}_1|-(\Theta\lambda)^{\sigma+3}t^{\sigma-1}=8(\Theta\lambda)^{\sigma+3}t^{\sigma-1}$. By \ref{lem:digraph}\ref{lem:digraph_a} applied to $D_1$ with $r=\Theta\lambda$ and $s=4(\Theta\lambda)^{\sigma+2}t^{\sigma-1}$, there is a stable set $\mca{S}_2$ in $D_1$ with
$|\mca{S}_2|=4(\Theta\lambda)^{\sigma+2}t^{\sigma-1}$. This proves \eqref{st:fewoutend}.

\sta{\label{st:digraphmajid} There is a $\Theta$-subset $\mca{P}$ of $\mca{S}_1$, a subpath $A_L$ of $L$ with $x_L\in V(A_L)$ for each $L\in \mca{P}$, and $\Theta$ pairwise disjoint $\lambda$-subsets $(\mca{R}_L:L\in \mca{P})$ of $\mca{S}_1\setminus \mca{P}$, such that $(A_L : L\in \mca{P})$ are pairwise anticomplete in $G$, and for each $L\in \mca{P}$ and every $R\in \mca{R}_L$, $A_L$ and $R\setminus x_R$ are not anticomplete in $G$.}

If $|D_1^{\geq \Theta\lambda}|\geq 2\Theta^2\lambda$, then \eqref{st:digraphmajid} follows from \eqref{st:manyoutend} and the fact that $\{x_L:L\in \mca{S}_1\}$ is a stable set in $G$. Assume that $|D_1^{\geq \Theta\lambda}|\leq 2\Theta^2\lambda$, which in turn yields $\mca{S}_2\subseteq \mca{S}_1$ as in \eqref{st:fewoutend}; in particular, $|\mca{S}_2|>\Theta\lambda$. Since no two paths in $\mca{S}_2\subseteq \mca{S}_0$ are anticomplete in $G$, for every $L\in \mca{S}_2$, there are at least $\Theta\lambda$ paths $L'\in \mca{S}_2\setminus \{L\}$ such that $L,L'$ are not anticomplete in $G$. For each $L\in \mca{S}_2$, traversing $L$ from $x_{L}$ to $y_L$, let $z_L$ be the first vertex for which there are at least $\Theta\lambda$ paths $L'\in \mca{S}_2\setminus \{L\}$ such that $x_L\dd L\dd z_L$ and $L'$ are not anticomplete in $G$, and let $A_L=x_L\dd L\dd z_L$. By \eqref{st:fewoutend}, $x_L$ is anticomplete to every path $L'\in \mca{S}_2\setminus \{L\}$. It follows that $x_L\neq z_L$ (and so $V(A_L\setminus z_L)\neq \varnothing$). Also, by the choice of $z_L$, there are fewer than $\Theta\lambda$ paths $L'\in \mca{S}_2\setminus \{L\}$ for which $A_L\setminus z_L$ and $L'$ are not anticomplete in $G$.

Let $D_2$ be the digraph with $V(D_2)=\mca{S}_2$ such that for all distinct $L,L'\in \mca{S}_2$, we have $(L,L')\in E(D_2)$ if and only if $A_L$ and $L'$ are not anticomplete in $G$. Then every vertex in $D_2$ has out-degree at least $\Theta\lambda$. Since $|\mca{S}_2|=4(\Theta\lambda)^{\sigma+2}t^{\sigma-1}$, by \ref{lem:digraph}\ref{lem:digraph_b} applied to $D_2$ with $q=\Theta$, $r=\lambda$ and $s=2(\Theta\lambda)^{\sigma+1}t^{\sigma-1}$, there is a $2(\Theta\lambda)^{\sigma+1}t^{\sigma-1}$-subset $\mca{S}_3$ of $\mca{S}_2$ with the following property: for every $\Theta$-subset $\mca{P}$ of $\mca{S}_3$, there are $\Theta$ pairwise disjoint $\lambda$-subsets $(\mca{R}_L:L\in \mca{P})$ of $\mca{S}_2\setminus \mca{S}_3$ such that for all $L\in \mca{P}$ and every $R\in \mca{R}_{L}$, $A_L$ and $R$ are not anticomplete in $G$. 

Let $D_3$ be the digraph with $V(D_3)=\mca{S}_3$ where for all distinct $L,L'\in \mca{S}_3$, we have $(L,L')\in E(D_3)$ if and only if $A_L\setminus z_L$ and $L'$ are not anticomplete in $G$. Then every vertex in $D_3$ has out-degree smaller than $\Theta\lambda$. Since $|V(D_3)|=|\mca{S}_3|=2(\Theta\lambda)^{\sigma+1}t^{\sigma-1}$, by \ref{lem:digraph}\ref{lem:digraph_a} applied to $D_3$ with $r=\Theta\lambda$ and $s=(\Theta\lambda)^{\sigma}t^{\sigma-1}$, there is a $(\Theta\lambda)^{\sigma}t^{\sigma-1}$-subset $\mca{S}_4$ of $\mca{S}_3$ such that $\mca{S}_4$ is a stable set in $D_3$. It follows that for all distinct $L,L'\in \mca{S}_4$, $A_L\setminus z_L$ and $L'$ are anticomplete in $G$, and in particular, $A_L\setminus z_L$ and $A_{L'}$ are anticomplete in $G$.

Now, since $|\mca{S}_4|>\Theta^{\sigma}t^{\sigma-1}$ and $G$ is $K_{\sigma,\sigma}$-free and $K_{t}$-free, by \Cref{lem:EHforKss}, there is a $\Theta$-subset $\mca{P}$ of $\mca{S}_4$ for which $\{z_L:L\in \mca{P}\}$ is a stable set in $G$. It follows that $(A_L: L\in \mca{P})$ are pairwise anticomplete in $G$. On the other hand, since $\mca{P}$ is a $\Theta$-subset of $\mca{S}_4\subseteq \mca{S}_3$, there are $\Theta$ pairwise disjoint $\lambda$-subsets $(\mca{R}_L: L\in \mca{P})$ of $\mca{S}_2\setminus \mca{S}_3\subseteq \mca{S}_2\setminus \mca{P}\subseteq \mca{S}_1\setminus \mca{P}$ such that for each $L\in \mca{P}$ and every $R\in \mca{R}_{L}$, $A_L$ and $R$ are not anticomplete in $G$. Recall that for all $L\in \mca{P}$ and $R\in \mca{R}_{L}$, $x_{R}$ is anticomplete to $V(L)$ in $G$ (because $R,L\in \mca{S}_2$). Thus, for all $L\in \mca{P}$ and $R\in \mca{R}_{L}$, $A_L$ and $R\setminus x_R$ are not anticomplete in $G$. This proves \eqref{st:digraphmajid}.
\medskip

Henceforth, let $\mca{P}$ and $(A_L,\mca{R}_L:L\in \mca{P})$ be as given by \eqref{st:digraphmajid}. Then, for each $L\in \mca{P}$ and every $R\in \mca{R}_L$, there is a vertex $w_R\in V(R\setminus x_R)$ such that $w_R$ is the only vertex of $w_R\dd R\dd y_R$ that has a neighbor in $G$ in $V(A_L)$. For every $L\in \mca{P}$, let
    $\mca{L}_L=\{w_{R}\dd R\dd y_{R}:R\in \mca{R}_L\}$ and let $Y_L=\{y_R:R\in \mca{R}_L\}$. Then, by \eqref{st:digraphmajid} and the choice of $(w_R:R\in \mca{R}_L)$, we have:

\sta{\label{st:almostseedling} For every $L\in \mca{P}$, $A_L$ is a path in $G[V(\mca{L})]$, we have $Y_L\subseteq V(\mca{L})\setminus V(A_L)$, and $\mca{L}_L$ is an $(N_{G[V(\mca{L})]}(V(A_{L})),Y_L)$-linkage in $G[V(\mca{L})]\setminus V(A_L)$ with $|\mca{L}_L|=\lambda$. Furthermore, the subsets $(V(A_L)\cup V(\mca{L}_L)\cup Y_L:L\in \mca{P})$ of
$V(\mca{L})$ are pairwise disjoint.}

We further claim that:

\sta{\label{st:seedlingrigid} There are distinct $L_1,\ldots, L_{\theta}\in \mca{P}$ such that for every $i\in \{1,\ldots, \theta\}$, $\mca{L}_{L_i}$ is $c_{\ref{lem:getbranches}}$-rigid.}

Suppose not. Since $|\mca{P}|=\Theta=\theta+2\mu\sigma-1$, there are distinct $P_1,\ldots, P_{2\mu\sigma}\in \mca{P}$ such that for every $i\in \{1,\ldots, 2\mu\sigma\}$, the $(N_{G[V(\mca{L})]}(V(A_{P_i})),Y_{P_i})$-linkage $\mca{L}_{P_i}$ in $G[V(\mca{L})]\setminus V(A_{P_i})$ is not $c_{\ref{lem:getbranches}}$-rigid. By definition, for $i\in \{1,\ldots, 2\mu\sigma\}$, there is an $(N_{G[V(\mca{L})]}(V(A_{P_i})),Y_{P_i})$-linkage $\mca{Q}_i=\{Q_{i,1},\ldots, Q_{i,c_{\ref{lem:getbranches}}}\}$ in $G$ with $V(\mca{Q}_i)\subseteq V(\mca{L}_{P_i})$ such that $Q_{i,1},\ldots, Q_{i,c_{\ref{lem:getbranches}}}$ are pairwise anticomplete in $G$. Also, recall that by \eqref{st:digraphmajid}, $A_{P_1},\ldots, A_{P_{2\mu\sigma}}$ are pairwise anticomplete in $G$, and note that by \eqref{st:almostseedling}, $(A_{P_i}, Q_{i,1}, \ldots, Q_{i,c_{\ref{lem:getbranches}}}: i\in \{1,\ldots, 2\mu\sigma\})$ are pairwise disjoint. Since $G$ has no induced minor isomorphic to $K_{\sigma,\sigma}$, by \Cref{lem:bigramsey} and the choice of $c_{\ref{lem:getbranches}}=c_{\ref{lem:bigramsey}}(\theta,1,\sigma)$, there is a $\mu$-subset $I$ of $\{1,\ldots, 2\mu\sigma\}$ and a path $B_i\in \mca{Q}_i$ for each $i\in I$, such that $(V(A_{P_i})\cup V(B_i):i\in I)$ are pairwise anticomplete in $G$. Recall that for each $i\in I$, $x_{P_i}\in V(A_{P_i})$ has a neighbor in $G$ in $V(A)$. Also, $B_i\in \mca{Q}_i$ has an end $x_i\in N_{G[V(\mca{L})]}(V(A_{P_i}))$ (so $x_i$ has a neighbor in $V(A_{P_i})$), and an end $y_i\in Y_{P_i}\subseteq Y$. Thus, there is a path $M_i$ in $G$ from $x_{P_i}\in N_G(V(A))$ to $y_i\in Y$ with $V(M_i)\subseteq V(A_{P_i})\cup V(B_i)\subseteq V(\mca{L})$. But now $\mca{M}=\{M_i:i\in I\}$ is a $(N_G(V(A)),Y)$-linkage in $G$ with $|\mca{M}|=\mu$, $V(\mca{M})\subseteq V(\mca{L})$, and $(M_i:i\in I)$ pairwise anticomplete, contrary to the $\mu$-rigidity of $\mca{L}$. This proves \eqref{st:seedlingrigid}.
\medskip

Let $L_1,\ldots, L_{\theta}$ be as in \eqref{st:seedlingrigid}. For $i\in \{1,\ldots, \theta\}$, define $A_i=A_{L_i}$, $\mca{L}_i=\mca{L}_{L_i}$ and $Y_i=Y_{L_i}$. By \eqref{st:almostseedling} and \eqref{st:seedlingrigid}, $(A_1,\mca{L}_1,Y_1),\ldots, (A_{\theta},\mca{L}_{\theta},Y_{\theta})$ are pairwise disjoint $(\lambda, c_{\ref{lem:getbranches}})$-seedlings in $G[V(\mca{L})]$. By \eqref{st:digraphmajid}, $A_1,\ldots, A_{\theta}$ are pairwise anticomplete in $G$, and for each $i\in \{1,\ldots, \theta\}$, $A$ and $A_i$ are not anticomplete in $G$, whereas $V(A)$ and $V(\mca{L}_i)$ are; because $x_{L_i}\in V(A_i)$ is the $N_G(V(A))$-end of $L_i\in \mca{L}$, and $V(\mca{L}_i)\subseteq V(\mca{R}_{L_i})\setminus \{x_R:R\in \mca{R}_{L_i}\}\subseteq V(\mca{L})\setminus N_G(V(A))$. This completes the proof of \Cref{lem:getbranches}.
\end{proof}

We also need two special cases of the Graham-Rothschild theorem \cite{GRgeneral}.

\begin{theorem}[Ramsey \cite{multiramsey}]\label{thm:multiramsey}
For all $r,s,t\in \poi$, there is a constant $c_{\ref{thm:multiramsey}}=c_{\ref{thm:multiramsey}}(r,s,t)\in \poi$ such that for every set $I$ with $0<|I|\leq r$, every set $A$ with $|A|\geq c_{\ref{thm:multiramsey}}$, and every function $\Phi:\binom{A}{s}\rightarrow I$, there exist $i\in I$ and a $t$-subset $B$ of $A$ such that $\Phi(S)=i$ for all $S\in \binom{B}{s}$.
\end{theorem}

\begin{theorem}[Graham, Rothschild, Spencer \cite{productramsey}]\label{thm:productramsey}
For all $r,s,t\in \poi$, there is a constant $c_{\ref{thm:productramsey}}=c_{\ref{thm:productramsey}}(r,s,t)\in \poi$ such that for every set $I$ with $0<|I|\leq r$, all sets $A_1,\ldots, A_s$ of cardinality at least $c_{\ref{thm:productramsey}}$, and every function $\Phi:A_1\times \cdots \times A_s\rightarrow I$, there exist $i\in I$ as well as a $t$-subset $B_j$ of $A_j$ for each $j\in \{1,\ldots, s\}$, such that $\Phi(b)=i$ for all $b\in B_1\times \cdots\times B_s$.
\end{theorem}

At long last, we can prove \Cref{thm:mainpolysep}, which we restate:

\mainpolysep*

\begin{proof}
Throughout, let $\rho,\sigma\in \poi$ be fixed, and let $G$ be as in the statement of \ref{thm:mainpolysep}. Let
$$\kappa=\kappa(\sigma)=c_{\ref{thm:multiramsey}}\left(8,3,3\sigma\right),\quad r=\binom{\kappa}{2},\quad s=2^{\binom{r}{2}},\quad \mu=\mu(\rho,\sigma)=c_{\ref{thm:productramsey}}\left(s,r,\sigma\right).$$
We begin with the following. (Similar arguments first appeared in \cite{pinned}, and later in \cite{tw17,tw18}.)

\sta{\label{st:getseedling} For every $\lambda\in \poi$, either there is a $(\lambda,\mu)$-seedling in $G$ or $G$ is $(\kappa,\lambda)$-separable.}

Suppose there is no $(\lambda,\mu)$-seedling in $G$, and also $G$ is not $(\kappa,\lambda)$-separable. Then there is a stable strong $(\kappa,\lambda)$-block $(B,\mca{P})$ in $G$. Let $\{x_1,y_1\},\ldots, \{x_r,y_r\}$ be an enumeration of all $2$-subsets of $B$ (recall the choice of $r$). For each $i\in \{1,\ldots, r\}$, let $\mca{L}_{i}=\{P^*:P\in \mca{P}_{\{x_i,y_i\}}\}$. Then $\mca{L}_{i}$ is an $(N_G(x_i),N_G(y_i))$-linkage in $G$ with $|\mca{L}_{i}|=\lambda$. Moreover, $x_i\notin N_G(y_i)$ because $(B,\mca{P})$ is stable. Since $(G[\{x_i\}],\mca{L}_{i},N_G(y_i))$ is not a $(\lambda,\mu)$-seedling in $G$, it follows that $\mca{L}_i$ is not $\mu$-rigid, and so there is an $(N_G(x_i),N_G(y_i))$-linkage $\mca{M}_{i}$ in $G$ with $|\mca{M}_i|=\mu$ and $V(\mca{M}_{i})\subseteq V(\mca{L}_{i})$ such that the paths in $\mca{M}_i$ are pairwise anticomplete in $G$. In particular, since $(B,\mca{P})$ is strong, $V(\mca{M}_{1}), \ldots, V(\mca{M}_{r})$ are pairwise disjoint. Let $\mca{I}$ be the power set of the set of all $2$-subsets of $\{1,\ldots, r\}$; thus $|\mca{I}|=s$. Define the function $\Phi_1:\mca{M}_1\times \cdots \times \mca{M}_{r}\rightarrow \mca{I}$ as follows: for every $(M_1,\ldots, M_r)\in \mca{M}_1\times \cdots \times \mca{M}_{r}$, let $\Phi_1(M_1,\ldots, M_r)\in \mca{I}$ be the set of all $2$-subsets $\{i,j\}$ of $\{1,\ldots, r\}$ for which $M_i$ and $M_j$ are not anticomplete in $G$. Since $|\mca{M}_1|=\cdots =|\mca{M}_{r}|=\mu$, by \Cref{thm:productramsey} and the choice of $\mu$, there exists $I\in \mca{I}$ as well as a $\sigma$-subset $\mca{M}'_i=\{M_{i,1},\ldots, M_{i,\sigma}\}$ of $\mca{M}_i$ for each $i\in \{1,\ldots, r\}$, such that $\Phi_1(M_1,\ldots, M_r)=I$ for every $(M_1,\ldots, M_r)\in \mca{M}'_1\times \cdots \times \mca{M}'_{r}$. But now for every $\{i,j\}\in I$, $(M_{i,1},\ldots, M_{i,\sigma}; M_{j,1},\ldots, M_{j,\sigma})$ is an induced $K_{\sigma,\sigma}$-model in $G$, which is a contradiction. Thus, $I=\varnothing$, and so $(M_{i,1}:1\leq i\leq r)$ are pairwise anticomplete in $G$. For each $i\in \{1,\ldots, r\}$, let $P_{\{x_i,y_i\}}=G[\{x_i,y_i\}\cup V(M_{i,1})]$. Then $P_{\{x_i,y_i\}}$ is a path in $G$ from $x_i$ to $y_i$ with $P^*_{\{x_i,y_i\}}=M_{i,1}$. Also, for all distinct $i,j\in \{1,\ldots, r\}$, we have $V(P_{\{x_i,y_i\}})\cap V(P^*_{\{x_j,y_j\}})=\varnothing$, and $P^*_{\{x_i,y_i\}}$ and $P^*_{\{x_j,y_j\}}$ are anticomplete in $G$. Let $B=\{v_1,\ldots, v_{\kappa}\}$ and for every $2$-subset $\{k,l\}$ of $\{1,\ldots, \kappa\}$, let $P_{\{k,l\}}=P_{\{v_k,v_l\}}$. Define the function $\Phi_2:\binom{\{1,\ldots, \kappa\}}{3}\to 2^{\{1,2,3\}}$ as follows: for every $3$-subset $S=\{k_1,k_2,k_3\}$ of $\{1,\ldots, \kappa\}$ with $k_1<k_2<k_3$, let $\Phi_2(S)$ be the set of all $\ell\in \{1,2,3\}$ for which $v_{k_{\ell}}$ is not anticomplete to $P^*_{S\setminus \{k_{\ell}\}}$ in $G$. By \Cref{thm:multiramsey} and the choice of $\kappa$, there exists $I\subseteq \{1,2,3\}$ and $\sigma$-subsets $K_1,K_2,K_3$ of $\{1,\ldots, \kappa\}$ with $\max K_1<\min K_2$ and $\max K_2<\min K_3$ such that for every $3$-subset $S$ of $K_1\cup K_2\cup K_3$, we have $\Phi_2(S)=I$ (note that $K_1,K_2,K_3$ are pairwise disjoint). Assume that $I\neq \varnothing$, say $\ell\in I\subseteq \{1,2,3\}$. Write $\{\ell',\ell''\}=\{1,2,3\}\setminus \{\ell\}$, $K_{\ell}=\{p_1,\ldots, p_{\sigma}\}$, $K_{\ell'}=\{q_1,\ldots, q_{\sigma}\}$, and $K_{\ell''}=\{r_1,\ldots, r_{\sigma}\}$. Then $\{v_{p_1},\ldots, v_{p_{\sigma}}\}$ is a stable set in $G$ because $B$ is, $P^*_{\{q_{1},r_{1}\}},\ldots, P^*_{\{q_{\sigma},r_{\sigma}\}}$ are pairwise anticomplete in $G$, and for all $i,j\in \{1,\ldots, \sigma\}$, $v_{p_{i}}$ is not anticomplete to $P^*_{\{q_j,r_j\}}$ in $G$ because $\ell\in I=\Phi_2(\{p_i,q_j,r_j\})$. But now $(G[\{v_{p_1}\}],\ldots, G[\{v_{p_{\sigma}}\}]; P^*_{\{q_{1},r_{1}\}}, \ldots, P^*_{\{q_{\sigma},r_{\sigma}\}})$ is an induced $K_{\sigma,\sigma}$-model in $G$, a contradiction. We deduce that $I=\varnothing$, and so for all pairwise distinct $k,k',k''\in K_1\cup K_2\cup K_3$, $v_{k}$ is anticomplete to $P^*_{\{k',k''\}}$ in $G$. Now, let $U=\bigcup_{\{k,k'\}\subseteq K_1\cup K_2}V(P_{\{k,k'\}})$. Then $G[U]$ is isomorphic to a proper subdivision of $K_{2\sigma}$, and $G$ has an induced minor isomorphic to $K_{\sigma,\sigma}$, a contradiction. This proves \eqref{st:getseedling}.
\medskip

Next, we define $\psi: \poi^3\rightarrow \poi\cup \{0\}$ recursively. Let $\psi(a,b,1)=b-1$ for all $(a,b)\in \poi^2$. For each $(a,b,c)\in \poi^3$ with $c\geq 2$, having defined $\psi(a',b',c-1)$ for all $(a',b')\in \poi^2$, let
$$\psi(a,b,c)=\sigma(\sigma+3)(2\sigma a+2\sigma b)(a-1)\left(\psi(c_{\ref{lem:getbranches}}(2\sigma b,\sigma), c_{\ref{lem:bigramsey}}(b,b,\sigma), c-1)+1\right).$$

We claim that:

\sta{\label{st:gettree} Let $a,b,c\in \poi$ and let $(A,\mca{L},Y)$ be a $(t^{\psi(a,b,c)},a)$-seedling in $G$. Then there exists a $(b,c)$-regular rooted tree $(T,u)$ and an induced $T$-model $(X_v:v\in V(T))$ in $G[V(A)\cup V(\mca{L})]$ with $X_u=A$.}

The proof is by induction on $c$ (for all $a,b\in \poi$). Assume that $c=1$. For each $L\in \mca{L}$, let $x_L$ be the $N_G(V(A))$-end of $L$. Since $G$ is $K_t$-free and $|\mca{L}|=t^{\psi(a,b,1)}=t^{b-1}$, by \Cref{thm:ramsey}, there are distinct $L_1,\ldots, L_{b}\in \mca{L}$ such that $\{x_{L_1},\ldots, x_{L_b}\}$ is a stable set in $G$. But now $(A;G[\{x_{L_1}\}],\ldots, G[\{x_{L_b}\}])$ is the desired model. Assume that $c\geq 2$.

Let $a'=c_{\ref{lem:getbranches}}(2\sigma b,\sigma)$, let $b'=c_{\ref{lem:bigramsey}}(b,b,\sigma)\geq b$ and let 
$\Psi=\psi(a', b', c-1)$. Then
$$t^{\psi(a,b,c)}=t^{\sigma(\sigma+3)(2\sigma a+2\sigma b)(a-1)(\Psi+1)}=\left(t^{\Psi}\cdot t\right)^{\sigma(\sigma+3)(2\sigma a+2\sigma b)(a-1)}.$$

Apply \Cref{lem:getbranches} to $(A,\mca{L},Y)$, $\theta=2\sigma b$, $\mu=a$ and $\lambda=t^{\Psi}$. Since $a'=c_{\ref{lem:getbranches}}(2\sigma b,\sigma)$, there exist $2\sigma b$ pairwise disjoint $(t^{\Psi},a')$-seedlings $(A_1,\mca{L}_1,Y_1),\ldots, (A_{2\sigma b},\mca{L}_{2\sigma b},Y_{2\sigma b})$ in $G[V(\mca{L})]$ such that $A_1,\ldots, A_{2\sigma b}$ are pairwise anticomplete in $G$, and for each $i\in \{1,\ldots, 2\sigma b\}$, $A$ and $A_i$ are not anticomplete in $G$ but $V(A)$ and $V(\mca{L}_i)$ are anticomplete in $G$. Since $\Psi=\psi(a',b',c-1)$, by the inductive hypothesis, there are $2\sigma b$ pairwise vertex-disjoint $(b',c-1)$-regular rooted trees $(T_1,u_1),\ldots, (T_{2\sigma b},u_{2\sigma b})$ such that for every $i\in \{1,\ldots, 2\sigma b\}$, there is an induced $T_i$-model $(X_{i,v}:v\in V(T_i))$ in $G[V(A_i)\cup V(\mca{L}_i)]$ with $X_{i,u_i}=A_i$. In particular, $V(A)$ and $\bigcup_{v\in V(T_i)\setminus \{u_i\}}V(X_{i,v})\subseteq V(\mca{L}_i)$ are anticomplete in $G$.

Let $i\in \{1,\ldots, 2\sigma b\}$ be fixed. Let
$u_{i,1},\ldots, u_{i,b'}$ be the children of $u_i$ in $(T_i,u_i)$ (as $c\geq 2$), and let $T_{i,1},\ldots, T_{i,b'}$ be the components of $T_i\setminus u_i$ containing $u_{i,1},\ldots, u_{i,b'}$, respectively. Then $(T_{i,1},u_{i,1}),\ldots, (T_{i,b'},u_{i,b'})$ are $(b',c-2)$-regular rooted trees, and since $b'\geq b$, for every $j\in \{1,\ldots, b'\}$, $(T_{i,j},u_{i,j})$ has a $(b,c-2)$-regular rooted subtree $(T'_{i,j},u_{i,j})$. For each $j\in \{1,\ldots, b'\}$, let 
$V_{i,j}=\bigcup_{v\in V(T'_{i,j})}V(X_{i,v})$ and let $A_{i,j}=G[V_{i,j}]$. 

Recall that $A_1,\ldots, A_{2\sigma b}$ are pairwise anticomplete paths in $G$. For each $i\in \{1,\ldots, 2\sigma b\}$, since $(X_{i,v}:v\in V(T_i))$ is an induced $T_i$-model in $G$ with $X_{i,u_i}=A_i$, it follows that $A_{i,1},\ldots, A_{i,b'}$ are pairwise anticomplete connected induced subgraphs of $G$, and for every $j\in \{1,\ldots, b'\}$, $A_i$ and $A_{i,j}$ are not anticomplete in $G$. Since $G$ has no induced minor isomorphic to $K_{\sigma,\sigma}$, and $b'=c_{\ref{lem:bigramsey}}(b,b,\sigma)$, by \Cref{lem:bigramsey} applied to $(A_i,A_{i,1},\ldots, A_{i,b'}:i\in \{1,\ldots, 2\sigma b\})$, there is a $b$-subset $I$ of $\{1,\ldots, 2\sigma b\}$ and a $b$-subset $J_i$ of $\{1,\ldots, b'\}$ for each $i\in I$, such that 
$(V(A_i)\cup (\bigcup_{j\in J_i}V_{i,j}):i\in I)$
are pairwise anticomplete in $G$.

For each $i\in I$, let
$U_i=\{u_i\}\cup (\bigcup_{j\in J_i}V(T'_{i,j}))\subseteq V(T_i)$. Then $((T_i[U_i],u_i):i\in I)$ are $b$ pairwise vertex-disjoint $(b,c-1)$-regular rooted trees. Let $T$ be a tree such that $V(T)=\{u\}\cup \left(\bigcup_{i\in I}U_i\right)$ and 
$E(T)=\{uu_i:i\in I\}\cup \left(\bigcup_{i\in I}E(T_i[U_i])\right)$, where $u\notin \bigcup_{i\in I}U_i$ is a new vertex. Then $(T,u)$ is a $(b,c)$-regular rooted tree, $(u_i:i\in I)$ are the children of $u$ in $(T,u)$, and for every $i\in I$, we have $T[U_i]=T_i[U_i]$. Let $X_u=A$. For each $i\in I$ and $v\in U_i$, let $X_v=X_{i,v}$; so $X_{u_i}=A_i$. Then, for every $i\in I$, $X_u$ and $X_{u_i}$ are not anticomplete in $G$, but $V(X_u)$ and $\bigcup_{v\in U_i\setminus \{u_i\}}V(X_v)\subseteq \bigcup_{v\in V(T_i)\setminus \{u_i\}}V(X_{i,v})$ are anticomplete in $G$. Moreover, $(X_{v}:v\in U_i)$ is an induced $T[U_i]$-model in $G$ with $X_{u_i}=A_i$ and 
$$\bigcup_{v\in U_i\setminus \{u_i\}}V(X_v)=\bigcup_{j\in J_i}\bigcup_{v\in V(T'_{i,j})}V(X_{i,v})=\bigcup_{j\in J_i}V_{i,j}.$$
In particular, it follows that $(\bigcup_{v\in U_i}V(X_v):i\in I)$ are pairwise anticomplete in $G$. Hence, $(X_v:v\in V(T))$
is an induced $T$-model in $G$ with $X_u=A$. This proves \eqref{st:gettree}.
\medskip

To finish the proof, we show that $G$ is $(c_{\ref{thm:mainpolysep}},t^{d_{\ref{thm:mainpolysep}}})$-separable, where
$$c_{\ref{thm:mainpolysep}}=c_{\ref{thm:mainpolysep}}(\sigma)=\kappa;\quad d_{\ref{thm:mainpolysep}}=d_{\ref{thm:mainpolysep}}(\rho,\sigma)=\psi(\mu,2,\rho).$$
Suppose not. Then by \eqref{st:getseedling} applied to $\lambda=t^{d_{\ref{thm:mainpolysep}}}$, there is a $(t^{d_{\ref{thm:mainpolysep}}},\mu)$-seedling in $G$. But now, since $d_{\ref{thm:mainpolysep}}=\psi(\mu,2,\rho)$, from \eqref{st:gettree} applied to $a=\mu$, $b=2$ and $c=\rho$, it follows that $G$ has an induced minor isomorphic to $\bin_{\rho}$, a contradiction. This completes the proof of \Cref{thm:mainpolysep}.
\end{proof}

\section{Acknowledgments}

The author thanks Sophie Spirkl for her support, encouragement, and helpful feedback (particularly for pointing out a mistake in an earlier version of \Cref{conj:cocks}). Thanks also to Sebastian Wiederrecht for asking a question that helped motivate this work, and to Maria Chudnovsky and Prashant Gokhale for discussions on related topics.
\bibliographystyle{plain}
\bibliography{ref}

\end{document}